\documentclass[11pt]{amsart}
\usepackage{amsmath}
\usepackage{setspace}
\usepackage{graphicx}
\usepackage{amssymb,amsthm}
\usepackage{hyperref}
\usepackage{esint}
\usepackage{color}
\usepackage{url}

\usepackage[top=1in, left=1in, right=1in, bottom=1in, marginpar=2cm]{geometry}

\newtheorem{theorem}{Theorem}[section]
\newtheorem{lemma}[theorem]{Lemma}
\newtheorem{proposition}[theorem]{Proposition}
\newtheorem{corollary}[theorem]{Corollary}

\newtheorem{remark}[theorem]{Remark}
\newtheorem{conjecture}[theorem]{Conjecture}
\newtheorem{definition}{Definition}[section]

\def\Z{{\mathbb Z}}
\def\R{{\mathbb R}}

\def\N{{\mathbb N}}
\def\C{{\mathbb C}}

\def\cH{{\mathcal H}}

\def\cL{{\mathcal L}}

\def\a{\alpha}
\def\b{\beta}
\def\e{\varepsilon}
\def\d{\delta}

\def\l{\lambda}

\def\n{\nabla}
\def\p{\partial}
\def\r{\rho}
\def\s{\sigma}
\def\t{\tau}

\def\w{\omega}
\def\W{\Omega}

\def\1{\left(}
\def\2{\right)}
\def\3{\left\{}
\def\4{\right\}}
\def\8{\infty}
\def\sm{\setminus}
\def\ss{\subseteq}
\def\cc{\subset\subset}

\newcommand{\mres}{\mathbin{\vrule height 1.6ex depth 0pt width
0.13ex\vrule height 0.13ex depth 0pt width 1.3ex}}

\DeclareMathOperator*{\dvg}{div}

\DeclareMathOperator*{\supp}{supp}

\usepackage{enumerate}

\newcommand{\maxN}{N^{\text{max}}}

\begin{document}

\title[Rectifiability and compactness for Bernoulli free boundaries]{Rectifiability, finite Hausdorff measure, and compactness for non-minimizing Bernoulli free boundaries}

\author{Dennis Kriventsov}
\address[Dennis Kriventsov]{Rutgers University, Piscataway, NJ}
\email{dnk34@math.rutgers.edu}

\author{Georg S. Weiss}
\address[Georg S. Weiss]{Faculty of Mathematics, University of Duisburg-Essen, Germany}
\email{georg.weiss@uni-due.de}

\begin{center}{\it Dedicated to David Jerison on the occasion of his 70th birthday.}\end{center}


\begin{abstract}
While there are numerous results on minimizers or stable solutions of the Bernoulli
problem proving regularity of the free boundary and analyzing singularities,
much less is known about {\em critical points} of the corresponding energy.
Saddle points of the energy (or of closely related energies) 
and solutions of the corresponding time-dependent problem  
occur naturally in applied problems such as water waves and
combustion theory.

For such critical points $u$---which can be obtained as limits of classical solutions or
limits of a singular perturbation problem---it has been open since \cite{W03} whether the singular set can be large and what equation the measure $\Delta u$ satisfies,
except for the case of two dimensions. In the present result we use recent techniques such as
a {\em frequency formula} for the Bernoulli problem as well as the
celebrated {\em Naber-Valtorta procedure} to 
answer this more than 20 year old question in an affirmative way:

For a closed class we call \emph{variational solutions} of the Bernoulli problem, we show that the topological free boundary $\partial \{u > 0\}$
(including \emph{degenerate} singular points $x$, at which 
$u(x + r \cdot)/r \rightarrow 0$ as $r\to 0$)
is countably $\cH^{n-1}$-rectifiable and
has locally finite $\cH^{n-1}$-measure,
and we identify the measure $\Delta u$ completely. This gives a more precise characterization of the free boundary of $u$ in arbitrary dimension than was previously available even in dimension two.

We also show that
limits of (not necessarily minimizing) classical solutions as well as 
limits of critical points of a singularly perturbed energy are
variational solutions, so that the result above applies directly to all of them.
\end{abstract}
\maketitle

\section{Introduction}
The one-phase Bernoulli problem is one the most studied in the free boundary literature. It has physical motivation from fluid dynamics and flame propagation, mathematical connections to minimal surfaces, optimization problems, and semilinear elliptic equations, and generally serves as an archetypal free boundary configuration. 

An elementary formulation of the problem is as follows: a nonnegative continuous function $u$ solves, on some domain $\W$,
\[
	\begin{cases}
		\Delta u = 0 & \text{ on } \{u > 0\} \cap \W, \\
		u = 0 & \text{ on } \p \{u > 0\} \cap \W, \\
		u_\nu = 1 & \text{ on } \p \{u > 0\} \cap \W,
	\end{cases}
\]
where $\nu$ is the inward unit normal to the (presumed sufficiently smooth) set $\{u > 0\}$. The problem, given a fixed set $\{u > 0\}$, is overdetermined, but as $\{u > 0\}$ is also free, there is hope of finding solutions. Indeed, one common approach is to observe that, formally, critical points to the functional
\begin{equation}\label{eq:altcafffunctional}
	E(u) = \int_\W \left(|\n u|^2 + \chi_{\{u > 0\}}\right)
\end{equation}
are solutions to this free boundary problem. Then \emph{minimizers} of \eqref{eq:altcafffunctional} will be critical points, and can easily be shown to exist (given some appropriate boundary conditions). A regularity theory for minimizers is by now fairly well-developed following \cite{AC}, and in particular implies that $\p \{u > 0\}$ is given by a smooth hypersurface separating $\{u > 0\}$ and the interior of $\{u = 0\}$ \cite{AC, KNS}, except on a set of high codimension (see \cite{W98, AC, CJK, JS}).

However, \eqref{eq:altcafffunctional} is not convex, and it is easy to see that there are critical points which are not minimizers (in any possible sense of the term). 
Actually, saddle points of the energy are more relevant than minimizers in some physical applications
such as water waves, and also are essential to studying the time-dependent problem with its own applications.

Although Alt and Caffarelli introduce the notion of {\em weak solutions}
in \cite{AC} and derive a regularity theory for them, limits of classical solutions
or of singular perturbations of the problem are in general
not contained in that class of weak solutions. In that sense 
the class of weak solutions
is too narrow.
On the other hand, the class of {\em viscosity solutions} is closed under limits but extremely large,
and our main result is likely false for general viscosity solutions (in light of the examples in \cite{wolff2}).
In Section \ref{sec:examples} we compare notions of solutions, including elementary examples.
 
In this paper we will be concerned with \emph{variational solutions} of \eqref{eq:altcafffunctional} (see Section \ref{sec:prelimvar}). These are, roughly, pairs of functions $(u, \chi)$ with $u \in C^0(\W) \cap C^2(\{u > 0\}\cap \W)$ nonnegative, Lipschitz continuous and harmonic when positive, while $\chi : \W \rightarrow \{0, 1\}$ generalizes the role played by the characteristic function $\chi_{\{u > 0\}}$. The pair must also be a critical point of \eqref{eq:altcafffunctional} under domain, or inner, variations. This is expressed by the identity
\[
	\int \left((|\n u|^2 + \chi) \dvg \xi - 2 \n u \cdot D \xi \n u\right) = 0
\]
for all vector fields $\xi \in C_c^\infty(\W; \R^n)$. Notice that without $\chi$, this is the traditional stationarity condition for the Dirichlet energy under perturbations of the form $u_t = u \circ \phi_t$, where $\phi_t$ is the flow of $\xi$; it is straightforward to verify that every ``classical'' solution of the Bernoulli problem is a variational solution with $\chi = \chi_{\{u > 0\}}$.

Variational solutions arise naturally from studying limits of more regular solutions, and were examined by the second author in \cite{W03} in the context of singular perturbations of semilinear equations. One observation made there is that while inner-variation solutions are closed under taking limits, they are closed only in a certain generalized sense. The limit $\chi$ of $\chi_{\{u_k > 0\}}$, which is part of the energy \eqref{eq:altcafffunctional}, is a characteristic function of a set, but it need not be the positivity set of the limiting function $\{u > 0\}$; instead, only the inclusion $\chi \geq \chi_{\{u > 0\}}$ holds, hence the relaxed definition here. 

The main result about variational solutions proved in \cite{W03} (stated here for elliptic rather than parabolic equations) is roughly as follows:
\begin{theorem}[Weiss \cite{W03}]\label{thm:weiss03}
	Let $(u, \chi)$ be a variational solution on $\W$. Then 
	\[
		\Delta u = \cH^{n-1}\mres \partial^*\{u>0\} + 2 \theta \cH^{n-1}\mres \Sigma_{**} + \l \mres \Sigma_z,
	\]
	where
	\begin{itemize}
		\item the reduced boundary $\partial^*\{u>0\}$ is open relative
to the topological free boundary and locally given by the graph of a smooth function.
		\item $\Sigma_{**}$ consists of points $x \in \p \{u > 0\}$ for which $u(y) = \theta |(y - x)_n| + o(|y - x|)$ locally after rotation, for some $\theta(x) > 0$; this set is countably $\cH^{n-1}$-rectifiable.
		\item $\Sigma_z$ is the set of {\em degenerate} free boundary points $x$, 
at which $u(x+r\cdot)/r\to 0$ as $r\to 0$, and $\lambda$ is a positive Borel measure with
		\[
			\lim_{r \to 0}\frac{\l(B_r(x))}{r^{n-1}} = 0
		\]
		for $\cH^{n-1}$-a.e. $x \in \W$.
	\end{itemize}
\end{theorem}
Note that the part 
$2 \theta \cH^{n-1}\mres \Sigma_{**}$ of the measure 
is not zero in general even for limits of classical solutions, and
$\theta$ can be in $(0,1)$ (see Section \ref{sec:examples}
as well as \cite[Introduction]{W03}).
However, the result does not settle the question of whether
$\l$ may be nonzero, or at least whether it may be supported on a large
set.
These questions are deeply related to the {\em harmonic measure}
of the free boundary, since for fixed $X\in \{u>0\}$, 
the harmonic measure $\omega^X$ of the set $\{u>0\}$ and $\Delta u$
are under mild assumptions on the set $\{u>0\}$ locally mutually absolutely continuous in the connected component
containing $X$.
So David Jerison and one of the authors
were in 2000 able to apply the result \cite{wolff} on plane harmonic measures
to obtain that the topological free boundary is a set
of $\sigma$-finite length and that in that case
 \[
		\Delta u = \cH^{1}\mres \partial^*\{u>0\} + 2 \theta \cH^{1}\mres \Sigma_{**}.
	\]
However, in dimension $n\geq 3$, the question whether the singular set 
(or even the whole topological free boundary)
can be a large
set and whether the measure $\lambda\ne 0$
has been completely open since.
Moreover, examples of extremely 
irregular harmonic measures in $\R^3$ (\cite{wolff2})
even suggested the contrary.

In this paper we give a positive answer to the questions above.
Our main theorem here is as follows:
\begin{theorem}\label{thm:intro}
	Let $(u, \chi)$ be a variational solution on a domain $\W$. Then:
	\begin{enumerate}[(i)]
		\item Either $u \equiv 0$ or $\chi = \chi_{\{u > 0\}}$.
        \item The whole topological free boundary $\partial \{u > 0\}$
is countably $\cH^{n-1}$-rectifiable and has locally finite $\cH^{n-1}$-measure.
        \item \[
		\Delta u = \cH^{n-1}\mres \p^* \{ u > 0\} + 2 \a(n) \sqrt{H_x(u; 0+)} \cH^{n-1}\mres \Sigma_{**}
		\]
		for $H_x(u; 0+) = \lim_{r\to 0}r^{-2}\fint_{\partial B_r(x)} u^2\> d\mathcal{H}^{n-1}$
		and some explicit normalizing constant $\a(n)$.
		\item At $\cH^{n-1}$-a.e. point $x \in \p\{ u > 0\}\sm\p^* \{ u > 0\}$,
		\[
			\frac{r\int_{B_r(x)}\left(|\n u|^2 + \chi - 1\right)}{\int_{\p B_r(x)}u^2 \> d\mathcal{H}^{n-1}} \rightarrow 1,
		\]
		and
		\[
			\frac{r \int_{B_r(x)} \left|\n (u - \a(n)|(y - x)\cdot \nu(x)|)\right|^2}{\int_{\p B_r(x)}u^2 \> d\mathcal{H}^{n-1}} \rightarrow 0
		\]
		as $r\to 0$, for some unit vector $\nu(x)$.
	\end{enumerate}
\end{theorem}
The first conclusion implies that variational solutions are closed under limits \emph{in the natural sense}, without tracking $\chi$ separately from $\chi_{\{u > 0\}}$---except in one specific case, where the limit is $u \equiv 0$ (and this cannot be improved, see the examples in \cite{W03}). The third conclusion shows that in the sense of distributions, $u$ satisfies the expected PDE, without the ``anomalous'' diffuse degenerate measure $\l$.

As we show in Section \ref{sec:compactness} that each limit of classical solutions of the
Bernoulli problem is a variational solution, we obtain as first application
of the Main Theorem above the following (see Theorem \ref{classcomp}).
\begin{theorem}\label{thm:compact_intro}
	Let $u_k$ be a sequence of classical solutions of the Bernoulli problem on $B_1$, with
	\[
		\sup_k \|u_k\|_{C^{0, 1}(B_1)} < \infty.
	\]
	Then, along a subsequence, $u_k \rightarrow u \in C^{0, 1}$ locally uniformly, $(u, \chi_{\{u > 0\}})$ is a variational solution, and for $\cH^{n-1}$-a.e. $x \in \Sigma_{**}(u)$,
	\[
		\a(n)\sqrt{H_x(u; 0+)} \leq 1.
	\]
	Moreover, either $u \equiv 0$ or $\chi_{\{u_k > 0\}} \rightarrow \chi_{\{u > 0\}}$ in $L^1_{\mathrm{loc}}(B_1)$ and $\p \{u_k > 0\} \rightarrow \p \{u > 0\}$ locally in $B_1$ in Hausdorff topology
as $k\to\infty$.
Most importantly,
$\p \{u > 0\}$ is countably $\cH^{n-1}$-rectifiable, has locally in $B_1$ 
finite $\cH^{n-1}$-measure, and
	\[
		\Delta u = \cH^{n-1}\mres \p^* \{ u > 0\} + 2 \a(n) \sqrt{H_x(u; 0+)} \cH^{n-1}\mres \Sigma_{**}.
	\]
\end{theorem}
Finally we show in Section \ref{sec:perturbation} that each limit of the related singular perturbation
problem for semilinear elliptic equations is a variational solution, so we obtain as second application
of Theorem \ref{thm:intro} above the following (see Theorem \ref{singpert}).
\begin{theorem}\label{thm:perturb_intro}
Let $(u_\epsilon)_{\epsilon\in (0,1)}$ be a uniformly bounded family of solutions
of \begin{equation}\label{epsprob0}
\Delta u_\epsilon=\beta_\epsilon(u_\epsilon) \> \textrm{ in } B_1,
\end{equation}
where
$\epsilon>0$,
$\beta_\epsilon(s)=\frac{1}{\epsilon}\beta(\frac{s}{\epsilon})$
and $\beta$ is a Lipschitz continuous function such that $\beta>0$ in
$(0, 1)$, $\beta\equiv 0$ outside $(0, 1)$ and
$\int_0^1\beta(s)\> ds=\frac{1}{2}$.
Then each limit $u$ as $\epsilon\to 0$ along a subsequence
satisfies the following:
$(u,\chi_{\{u>0\}})$ is a variational solution,
 $\p \{u > 0\}$ is countably $\cH^{n-1}$-rectifiable, has locally in $B_1$ 
finite $\cH^{n-1}$-measure, and
	\[
		\Delta u = \cH^{n-1}\mres \p^* \{ u > 0\} + 2 \a(n) \sqrt{H_x(u; 0+)} \cH^{n-1}\mres \Sigma_{**};\]
moreover, $\a(n) \sqrt{H_x(u; 0+)}\leq 1$ for $\cH^{n-1}$-a.e. $x\in \Sigma_{**}$. 
\end{theorem}

The proof of the Main Theorem \ref{thm:intro} relies on a 
{\em frequency formula on the set of free boundary points of highest density}
\[
	\Sigma^H = \{x\in \p\{ u>0\}: \lim_{r\to 0} \frac{|B_r\cap \{\chi = 1\}|}{|B_r|} = 1\}
\]
(see Lemma \ref{lem:frequency} and Proposition \ref{prop:Mlimits}).
The monotonicity of our frequency function is not a perturbative phenomenon based on $u$ being approximately harmonic near such points, but rather is exact and incorporates the full nonlinear structure of the problem. Ideas concerning this frequency (but not the monotonicity per se) already appear in \cite{W03}, and the monotonicity and other properties are established for related problems in \cite{VW}.

One key difficulty with using this frequency is that rescalings of a variational solution $u$ of the form $$u_{x, r}(y) = \frac{u(x + r y)}{\sqrt{\fint_{\p B_r} u^2} },$$ are not variational solutions: the Bernoulli problem is invariant under one-homogeneous rescalings but not these. As such, the only estimates available on them come from the monotonicity of the frequency, and in particular these are not sufficient to ensure convergence to a limit strongly in $W^{1, 2}$ or (roughly equivalently) to deduce that the limit is a ``variational solution'' of the Laplace equation. It is then unclear how to characterize the possible tangent objects at points in $\Sigma_z$.

The second main ingredient in this paper is
a new, quantitative, method from geometric measure theory pioneered by Naber and Valtorta in \cite{NV} which allows obtaining considerable information about $\Sigma^H$ by only looking at points and scales where $u_{x, r}$ is ``approximately one-dimensional,'' in the sense that $\p_e u_{x, r}$ is small in $n-1$ directions $e$. In this specific case, together with information derived from the quantitative monotonicity of the frequency, we are able to establish that $u_{x, r}$ converges strongly in $W^{1, 2}$ to a specific one-dimensional profile of the form $\a(n) |y \cdot \nu|$ for some unit vector $\nu$, which has frequency $1$ (the minimum possible frequency in this context). This is enough to justify a dichotomy crucial to the Naber-Valtorta argument: at any scale, either all points in $\Sigma^H$ already have frequency close to $1$, or the frequency between scale $r$ and $\e r$ decreases by a large amount outside of a set which is approximately $n-2$ dimensional at most. The point is that if the frequency does not change much on an approximately $n-1$-dimensional set, then we can show $u_{x, r}$ is approximately one-dimensional and so is close to a well-behaved tangent object. Alongside other technical ingredients, this is the main idea of the proof of Theorem \ref{thm:intro} (ii), while the others follow as corollaries after some further analysis.

In most previous works along these lines, such as for nodal sets of elliptic equations \cite{NV2} or harmonic maps \cite{NV}, the compactness argument required to prove the frequency drop dichotomy is elementary, while for us it poses a core challenge. In a recent related theorem about $Q$-valued harmonic functions
\cite{DNSV}, the authors also face concentration compactness difficulties when passing to
``frequency limits.''
However, the fact that Almgren's frequency formula holds everywhere while ours
holds only on the set of highest density already points to crucial differences
between the two problems: the presence of the regular set in the Bernoulli
problem makes for an (additional) source term in the equation 
which has no counterpart in the $Q$-valued harmonic function problem.
Another difference is that in the $Q$-valued harmonic function problem,
even continuity seems to pose a problem while continuity is in this paper
on the Bernoulli problem part of the assumptions, motivated by the fact that
limits of critical points of the singularly perturbed problem are
(by a Bernstein property) locally Lipschitz continuous.
Note, however, that {\em uniform} continuity of the scaled functions
$u_{x,r}$ is unknown.  
Incidentally, the question whether {\em all} domain variation solutions are Lipschitz
continuous is an interesting question for future research (cf. Remark \ref{rem:lip}).  

The paper is organized as follows: Section \ref{sec:examples} gives a number of elementary examples to better explain the notion of variational solution and the set $\Sigma^H$. In Section \ref{sec:prelim} we recall the notion of variational solution and discuss their basic properties and the monotonicity of the local energy; this follows \cite{W03} but in a simpler context. In Section \ref{sec:freq}, the frequency formula and its consequences are developed. This section follows \cite{VW} for the most part, but again for a simpler problem. In Section \ref{sec:blowup}, we classify limits for approximately one-dimensional frequency blow-up sequences, which play a key role in the rest of the arguments. Section \ref{sec:estimate} contains two major estimates: the frequency drop dichotomy described above, and the $L^2$-subspace approximation bound (which controls the size and rectifiability properties of the large, i.e. at least approximately $n-1$-dimensional, portions of $\Sigma^H$ by drops in frequency). Then in Section \ref{sec:NV}, these are combined with covering arguments and the Naber-Valtorta Reifenberg theorems to prove Theorem \ref{thm:intro} (ii). The rest of the theorem is proved in Section \ref{sec:corollaries}.
Section \ref{sec:compactness} presents a simple example compactness theorem for limits of ``classical'' solutions to the Bernoulli problem, making no attempt at maximal generality but already seemingly obtaining a result quite different from what is available in the literature.
Finally Section \ref{sec:perturbation}
shows that the conclusions of Theorem \ref{thm:intro} hold for limits of the 
related singular perturbation problem.
\section{Notation}
Throughout this work $\R^n$ will be equipped with the Euclidean inner product $x \cdot y$ and the induced norm $|x|$. Due to the nature of the problem we will sometimes write $x \in \R^n$ as $x=(x', x_n)$. $B_r(x)$ will be the open $n$-dimensional ball of center $x$ and radius $r$. Whenever the center is omitted it is assumed to be $0$.
The function $d(x,A)$ shall denote the Euclidean distance of the point $x\in \R^n$ to the set $A\subset \R^n$.
Moreover, we will use the negative part of a function, $u_-:=\max(-u,0)$.

When considering a set $A$, $\chi_A$ shall denote the characteristic function of $A$.
$\cH^{k}$ is the $k$-dimensional Hausdorff measure
and $\mathcal{L}^n$ the $n$-dimensional Lebesgue measure,
where we denote $\mathcal{L}^n(A) =: |A|$.
\section{Preliminaries} \label{sec:prelim}
\subsection{Examples}\label{sec:examples}

We would like here to provide some examples of solutions, in various generalized senses, to the Bernoulli problem. First, it may be useful to describe several possible notions of solution, alongside trivial one-dimensional examples which already demonstrate the differences between them (though they do not truly capture the much more delicate behavior of the free boundary in higher dimensions which makes the problem challenging). Below we examine possible solutions on an interval $[-1, 1]$.
\begin{itemize}
	\item Given two values $u(\pm 1)$, there is always a minimizer to \eqref{eq:altcafffunctional}. The minimizer is either the harmonic (linear) function $u(t) = \frac{1}{2}[u(1)(1 + t) + u(-1)(1 - t)]$, or possibly $u(t) = (u(-1) - 1 - t)_+ + (u(1) + t - 1)_+$ if $u(1) + u(-1) \leq 2$. Which of these is the minimizer depends on the parameters; for example, if $u(1) = u(-1) = a$, then the first is the only minimizer when $a > \frac{1}{2}$, the second is the only minimizer when $a < \frac{1}{2}$, and they are both minimizers if $a = \frac{1}{2}$. 
	Note that the class of minimizers is compact, i.e. limits of minimizers are again minimizers.
	\item A second notion is that of a \emph{weak}, \emph{distributional}, or \emph{outer-variation} solution. The exact meaning of this differs between authors in the higher-dimensional context (see, e.g., the weak solutions in \cite{AC}), but the basic premise is that it should satisfy, for any smooth compactly supported function $\eta$,
	\[
		E((u + t \eta)_+) = E(u) + o(t).
	\]
At $x \in \p \{u > 0\}$, in one dimension it is easy to see that $u(t) = (t - x)_+$ or $u(t) = (t - x)_-$ locally. The main problem with this property is that it is not closed under limits: the function $|t|$ is not a weak solution in this sense, but is a limit of the functions $(|t|-\e )_+$ which are.
	\item We may instead consider \emph{inner-variation}, \emph{domain-variation}, or simply \emph{variational} solutions, which instead have that for any smooth vector field $\xi \in C_c^\infty(\W)$, if we consider the flow $\phi_t$, then
	\begin{equation}\label{inner}
		E(u \circ \phi_t) = E(u) + o(t).
	\end{equation}
	Here we will separately also assume that $u$ is smooth when positive (which is not implied by \eqref{inner}). With this assumption, like the outer-variation property,
it follows that $u$ is harmonic when positive. According to the shape of $\{u > 0\}$, there are two cases: either $u$ is locally $u(t) = (t - x)_+$ or $u(t) = (t - x)_-$ as before, or $u(t) = \a |t - x|$ locally for some $\a > 0$. Unlike outer-variation solutions, these are closed under limits, but in exchange we are forced to settle for a seemingly very weak ``balancing'' condition at 
``multiplicity $2$ points'' where $u(t) = \a |t - x|$.
	\item An entirely different framework exists based on the maximum principle, called \emph{viscosity solutions}. While we will revisit these in Section \ref{sec:corollaries}, for the moment it is easiest to state the criterion for a one-dimensional $u$ which is linear when positive to be a viscosity solution. For $x \in \p \{u > 0\}$, $u$ must locally be either $u(t) = (t - x)_+$ or $u(t) = (t - x)_-$, or of the form $u(t) = \a (t - x)_+ + \b (t - x)_-$ with $\a, \b \in (0, 1]$. Viscosity solutions are also closed under taking limits, but the ``supersolution'' property $\a, \b \leq 1$ is  much weaker than the balancing condition of variational solutions, and in higher dimensions it becomes difficult to say anything about general viscosity solutions at points in $\p \{u > 0\}$ with Lebesgue density $1$ for $\{u > 0\}$.
\end{itemize}
Notice that there is little difference between these notions at points where the Lebesgue density of $\{u > 0\}$ is strictly less than $1$.
At points in $\Sigma^H$, however, the differences are considerable. Minimizers simply do not possess any such points, as was observed in \cite{AC}. This fact and the already mentioned
example of two smooth interfaces approaching each other and touching in a limit,
$(|x_n|-\epsilon)_+$, suggest an analogy to {\em multiplicity-$2$ points in minimal surfaces}.
Outer-variation solutions appear to exclude $\Sigma^H$-points entirely (this depends on the specific notion), leading to obvious instability under limits. 
In fact, as shown by the last example of this subsection, 
the ``natural'' condition that $u_{\pm \nu} = 1$ on each side, i.e. $u(t) = |t - x|$ locally, is 
not preserved when passing to limits. 
The variational solutions require a weaker balancing condition which is enough to obtain monotonicity of local energies similar to the minimizer case, and the purpose of this paper is to show that they are fairly regular and stable under limits in a strong sense. 
We conjecture that there are examples of viscosity solutions such that
the topological free boundary is not a set of $\sigma$-finite $\cH^{n-1}$-measure.

Turning now to variational solutions (see the formal definition in Section \ref{sec:prelimvar} below), we present some simple examples to illustrate what kind of structure may be possible for $\Sigma^H$.
\begin{itemize}
	\item If $\chi_{\{u > 0\}} = 1$ a.e., then so does $\chi$, and the variational identity simplifies to the one for the Dirichlet energy alone:
	\begin{equation}\label{eq:variationalformulaharmonic}
		\int \left(|\n u|^2 \dvg \xi - 2 \n u \cdot D \xi \n u\right) = 0.
	\end{equation}
	Any harmonic function, of course, satisfies this identity. Given a sign-changing harmonic function $v$, though, so does $u = |v|$, as $\n u = \pm \n v$ almost everywhere. This class of solutions already shows that Theorem \ref{thm:main}
implies strong nodal set estimates for harmonic functions \cite{L, NV2}, and gives examples where the degenerate singular set $\Sigma_z$ is nonempty (like $u = |\Re (z^2)|$ on $\C = \R^2$). Such examples suggest that $\Sigma_z$ should be of codimension two, but we leave this point to future research.
	\item Given any harmonic function $v$ and number $a$, $u = |v| + a$ satisfies \eqref{eq:variationalformulaharmonic} (and many other examples of this general type are possible). However, these are not variational solutions, as they are not in $C^2(\{u > 0\})$. The assumption 
$u\in C^2(\{u > 0\})$ guarantees that $u$ is harmonic on $\{u > 0\}$, which otherwise does not follow from \eqref{eq:variationalformulaharmonic} alone; we do not attempt to find the minimal assumptions required for this here.
	\item There are other solutions to \eqref{eq:variationalformulaharmonic} not of this type. For example, take $u(z) = |\Re (z^\a)|$; this is a variational solution for $\a = \frac{n}{2}$, $n \in \N$ with $n \geq 2$ ($n = 1$ is excluded by the assumption that $u$ must be Lipschitz). However, $\Re z^\a$ is harmonic only for $n$ even; the odd $n$ instead come from two-valued harmonic functions. The theory of multivalued harmonic functions is quite subtle and has received considerable attention recently (see \cite{DNSV} as well as the references therein). Indeed, our work was partly motivated by the approach in \cite{DNSV}.
	\item Explicit examples including both regular and singular points 
include for example {\em cusp} singularities such as the two-dimensional example (see \cite{AC})
$$u(x)=\max(-\log(|x-e_1|),0)+\max(-\log(|x+e_1|),0)$$
in the domain $\Omega = \R^2\setminus \left(B_{1/2}(e_1)\cup B_{1/2}(-e_1)\right)$.
	\item In \cite{BSS}, a family of periodic (classical) solutions in the plane is constructed, for which $\{u > 0\}$ consists of the complement of $\bigcup_{j \in \Z}(j e_1 + K)$, where $K$ is a closed convex set symmetric about both axes and contained in $\{|x_1| < \frac{1}{2} \}$. All solutions of this type consist of a single one-parameter family $u = u_\alpha$, where the parameter $\alpha \in (0, 1)$ is half the perimeter of $K$ (i.e. for each half-perimeter $\alpha \in (0, 1)$ there is a unique solution, while for $\alpha \geq 1$ there are no solutions), and these solutions satisfy $|\nabla u| < 1$ on $\{u > 0\}$ (see \cite{Traizet}). Fix a value of $\alpha$, and consider the blow-down limit $u_R = u(R \cdot)/R$, as $R \rightarrow \infty$: it is straightforward to check using Liouville's theorem that $u_R \rightarrow q |x_2|$ for some $q \in [0, 1]$. Moreover, $q = \alpha$: indeed, integrating $\Delta u$ over the region $A = ([-\frac{1}{2}, \frac{1}{2}] \times [0, R])\setminus K$ and applying the divergence theorem gives
	\[
		0 = \int_A \Delta u = \int_{\partial K \cap \{x_2 > 0\}} u_\nu + \int_{[-\frac{1}{2}, \frac{1}{2}] \times \{R\}} u_n \rightarrow - \alpha + q.
	\]
	In particular, this implies that the variational solutions $\alpha |x_2|$ are for each $\alpha \in (0,1)$ explicitly obtainable as limits of entire classical solutions, and so the conclusion of Theorem \ref{thm:compact_intro} cannot be improved to read $\a(n)\sqrt{H_x(u; 0+)} = 1$ instead. We note here that Traizet \cite{Traizet} proved that periodic solutions of this type are in bijective correspondence with specific types of minimal surfaces, and derived even stronger rigidity theorems based on this property.
\end{itemize}
\subsection{Variational solutions}\label{sec:prelimvar}

Given an open set $\W \ss \R^n$, we say that a pair of functions $(u, \chi)$ with $u : \Omega \rightarrow [0, \infty)$ and $\chi : \Omega \rightarrow \{0, 1\}$ is \emph{a variational solution} (with constant $C_V$) if the following properties hold:
\begin{enumerate}
	\item $u \in C^0(\W)\cap C^2(\{u > 0\})$ with $|\nabla u|\leq C_V$ on $\W$.
	\item $\chi$ is a Borel measurable function.
	\item $\chi_{\{u > 0\}} \leq \chi$ $\cL^n$-a.e. on $\W$.
	\item For each vector field $\xi \in C^{0, 1}_c(\W; \R^n)$, 
	\begin{equation}\label{eq:variationalformula}
		\int \left((|\nabla u|^2 + \chi) \dvg \xi - 2 \nabla u \cdot D\xi \nabla u \right)= 0.
	\end{equation}
\end{enumerate}
\begin{remark}\label{rem:lip}
	\begin{enumerate}
		\item The assumption $u \in C^2(\{u > 0\})$	can definitely be reduced to some milder assumption.
		\item We conjecture that the Lipschitz assumption cannot be omitted,
		that is, there are examples of solutions of \eqref{eq:variationalformula}
		which are not Lipschitz continuous.
	\end{enumerate}
\end{remark}

Variational solutions have a simple scaling property: if $(u, \chi)$ is a variational solution on $B_r(x)$, then setting
\begin{equation} \label{eq:onehomrescale}
	\begin{cases}
		u_{x, r}(y) &= \frac{u(x + r y)}{r}, \\
		\chi_{x, r}(y) &= \chi(x + r y)
	\end{cases}
\end{equation}
implies that $(u_{x, r}, \chi_{x, r})$ is a variational solution on $B_1$ (with the same constant). We will refer to such a change of variables as a \emph{one-homogeneous rescaling} below.

\begin{proposition}\label{prop:subharmonic}
	Let $(u, \chi)$ be a variational solution on $\W$. Then $u$ is subharmonic on $\W$, and harmonic on $\{u > 0\}$.
\end{proposition}

\begin{proof}
	Take $\xi \in C_c^\infty(\{u > 0\}; \R^n)$. Then $ \chi \geq \chi_{\{u > 0\}} = 1$ a.e. on the support of $\xi$, so integrating by parts gives
	\[
		\int \left(|\nabla u|^2 \dvg \xi - 2 \nabla u \cdot D \xi \nabla u\right) = 0.
	\]
	As $u\in C^2(\Omega\cap \{u>0\})$, 
integration by parts yields 
that 
$\int  \xi \cdot \nabla u \> \Delta u= 0$, which in turn gives that 
$u$ is harmonic in $\{u>0\} \cap \{\nabla u \neq 0\}$. Hence $\Delta u = 0$ on the closure $\{u>0\} \cap \overline{\{\nabla u \neq 0\}}$, and so on $\{u > 0\}$, using that $u \in C^2(\{u > 0\})$ by assumption. It then follows that 
the continuous function
$u$ is subharmonic on $\W$, for example by verifying it is subharmonic in the viscosity sense.
\end{proof}

\begin{lemma}\label{lem:BV}
	Let $(u, \chi)$ be a variational solution on $\W$, and $B_{2r}(x) \cc \W$. Then $\chi \in BV(B_r(x))$, and
	\[
		\int_{B_r(x)} |\n \chi| \leq C r^{n-1},
\text{ where the constant }C \text{ depends only on }C_V.
	\]
\end{lemma}

\begin{proof}
	We have the following elementary estimate on the Laplacian of $u$ (recall that this is represented by a non-negative Borel measure): for a cutoff function $\eta$ such that $\eta\equiv 1$ on $B_{3r/2}(x)$ and $\eta\equiv 0$ outside of $B_{2r}(x)$,
	\[
		\Delta u(B_{3r/2}(x)) \leq \int \eta \> d\Delta u = - \int \n u \cdot \n \eta \leq C r^{n-1}.
	\]
	Let $\phi_t$ be a standard mollifier, and $u_t = \phi_t * u$. Then $u$ is smooth, and we have that from \eqref{eq:variationalformula},
	\[
		\lim_{t } \int \left(|\n u_t|^2 \dvg \xi - 2 \n u_t \cdot D\xi \> \n u_t\right) = \lim_{t \to 0+} \int \left(|\n u|^2 \dvg \xi - 2 \n u \cdot D\xi \n u\right) = - \int \chi \dvg \xi
	\]
	for every $\xi \in C_c^{0, 1}(B_{3r/2}(x); \R^n)$. We can rewrite this using integration by parts:
	\begin{align*}
		\int \left(|\n u_t|^2 \dvg \xi - 2 \n u_t \cdot D\xi \> \n u_t\right) & = \int \left( \dvg(|\n u_t|^2 \xi - 2 \n u_t \cdot \xi \> \n u_t) + 2 \n u_t \cdot \xi \> \Delta u_t\right) 
\\& = 2 \int \n u_t \cdot \xi \> \Delta u_t.
	\end{align*}
	We know that $|\n u_t|\leq C_V$, so
	\[
		\left|\int \left(|\n u_t|^2 \dvg \xi - 2 \n u_t \cdot \dvg \xi \> \n u_t\right)\right| \leq C(C_V) \sup |\xi| r^{n-1}.
	\]
	In particular,
	\[
		\left|\int \chi \dvg \xi\right| \leq C(C_V) \sup |\xi| r^{n-1},
	\]
	which implies that $\chi \in BV(B_r(x))$ and also the estimate.
\end{proof}

\subsection{The monotonicity formula}\label{sec:monotone}

Given a variational solution on $\W$ with $B_r(x)\ss \W$, we set
\[
	D_x(u; r) = \frac{1}{r^n}\int_{B_r(x)} \left(|\nabla u|^2 + \chi\right)
\]
and
\[
	H_x(u; r) = \frac{1}{r^{n+1}}\int_{\partial B_r(x)} u^2\> d\mathcal{H}^{n-1}.
\]
Define the \emph{monotonicity formula} to be the difference,
\[
	M_x(u; r) = D_x(u; r) - H_x(u; r).
\]
When there is no ambiguity we will drop the parameter $u$, using e.g. $M_x(r) = M_x(u; r)$.

Both $H$ and $D$ are invariant under the one-homogeneous rescaling \eqref{eq:onehomrescale}, in the sense that $H_0(u_{x, r}; t) = H_x(u; t r)$ (and similarly for $D$, and so the monotonicity formula $M$).
\begin{proposition}[{\cite[Theorem 3.1 and Theorem 4.1]{W98}}]\label{prop:monotone}
	Let $(u, \chi)$ be a variational solution on $B_r(x)$. Then for $s \in (0, r)$, $M_x(s)$ is an absolutely continuous, nondecreasing function of $r$, with
	\[
		\frac{d}{dr}M_x(s) = \frac{2}{s^{n+2}}\int_{\partial B_s(x)} (u(y) - (y - x) \cdot \nabla u(y))^2 \> d\cH^{n-1}(y)
	\]
	for almost every $s$.
The limit $M_x(0+) = \lim_{r \to 0} M_x(r) \in [-\infty, \infty)$ exists, and for each open
$D\subset\subset \R^n$, $(u_{x,r})_{r\in (0,\delta(D))}$ is bounded in $W^{1,2}(D)$
and each limit with respect to a subsequence is a homogeneous function of degree $1$. 
\end{proposition}

\begin{proof}
As some of the identities of the proof will be used later, we will give the proof of the monotonicity.

	It is straightforward to check that both of $D_x, H_x$ are Lipschitz in $s$, using only that $u$ is a Lipschitz function; in particular, they are absolutely continuous.
	
	We compute the derivatives (assuming, after a translation, that $x = 0$), by first using $ \xi(y) = y \eta(|y|)$ in the definition of variational solution, where $\eta \in C^{0, 1}([0, r))$, $\eta(r) = 0$ is a cutoff function to be selected below. This leads to
	\[
		0 = \int \left((|\n u|^2 + \chi) (n \eta(|y|) + |y| \eta'(|y|)) - 2 \n u \cdot [ \eta(|y|)\n u + \frac{y}{|y|} \eta'(|y|) y \cdot \n u]\right).
	\]
	Select a sequence of $\eta_k$ with $\eta_k(t) = 1$ for  $t \in [0, s]$, $\eta_k(t) = 0$ for $t > s + 1/k$, and $\eta_k$ linear in between. Then
	\[
		\int_{B_s} (n - 2) (|\n u|^2 + \chi) = k \int_s^{s + 1/k}\int_{S^{n-1}}\left( t (|\n u|^2 + \chi) - 2 t (\n u \cdot \w)^2 \right)d\w \> t^{n-1}dt + O(k^{-1}),
	\]
	and at every Lebesgue point of $s \mapsto \int_{\partial B_s} \left(|\n u|^2 + \chi - 2 (\n u \cdot \frac{y}{|y|})^2 \right)\> d\mathcal{H}^{n-1}(y)$ the integral on the right passes to its limiting value, leading to
	\begin{equation}\label{eq:rellichpohozaev}
		(n - 2)\int_{B_s} |\n u|^2 + n\int_{B_s} \chi = s \int_{\partial B_s}|\n u|^2 + \chi - 2 (\n u \cdot \frac{y}{|y|})^2 \> d\cH^{n-1}(y)
	\end{equation}
	for almost every $s$. On the other hand, we have that
	\[
		\int_{B_s}|\n u|^2 = \lim_{\e\to 0+}\int_{B_s}\n u \cdot \n (u - \epsilon)_+ = \lim_{\e\to 0+} \int_{\partial B_s} (u - \epsilon)_+ \n u \cdot \frac{y}{|y|}\> d\cH^{n-1}(y)
	\]
	using the dominated convergence theorem and integrating by parts (for a sequence of $\e$ for which $\partial \{u > \epsilon\}$ is smooth). Passing to the limit leads to
	\begin{equation}\label{eq:intbyparts}
		\int_{B_s}|\n u|^2 = \int_{\partial B_s} u \n u \cdot \frac{y}{|y|}\> d\cH^{n-1}(y)
	\end{equation}
	for almost every $s$. Plugging this into \eqref{eq:rellichpohozaev} then gives
	\begin{equation}\label{eq:Didentity}
		n D_0(s) = \frac{1}{s^{n-1}} \int_{\partial B_s}\left(|\n u|^2 + \chi - 2 (\n u \cdot \frac{y}{s})^2 + 2 u \n u \cdot \frac{y}{s^2}\right)\> d\cH^{n-1}(y)
	\end{equation}

	Applying Fubini's theorem in spherical coordinates to $s^n D_0(s)$, we see that
$D_0$ is absolutely continuous. Differentiating $D_0(s)$ gives, at almost every $s$,
	\begin{align*}
		\frac{d}{ds}D_0(s) &= \frac{1}{s^n} \int_{\partial B_s} \left(|\nabla u|^2 + \chi\right)\> d\mathcal{H}^{n-1} - \frac{n}{s} D_0(s) \\
		& = \frac{2}{s^{n + 2}} \int_{\partial B_s}\left((\n u \cdot y)^2 - u \n u \cdot y \right)\> d\cH^{n-1}(y)
	\end{align*}
	by using \eqref{eq:Didentity}.

	Differentiating $H_0(s)$ gives
	\begin{align*}
		\frac{d}{ds}H_0(s) &= \frac{n - 1}{s}H_0(s) - \frac{n + 1}{s}H_0(s) + \frac{2}{s^{n+1}}\int_{\partial B_s} u \n u \cdot \frac{y}{|y|}d \cH^{n-1}(y) \\
		&= -\frac{2}{s} H_0(s) + \frac{2}{s^{n+2}}\int_{\partial B_s} u \n u \cdot yd \cH^{n-1}(y).
	\end{align*}
	Putting the two together leads to
	\[
		\frac{d}{ds}M_0(s) = \frac{2}{s^{n+2}} \int_{\partial B_s} ( u - \n u \cdot y)^2 \> d\cH^{n-1}(y) \geq 0.
	\]
\end{proof}
 
\begin{proposition}\label{prop:Mlimits}
	Let $(u, \chi)$ be a variational solution on $\W$, and $x\in \W$. Then $M_x(u; 0+) = - \infty$ if and only if $u(x) > 0$. If $u(x) = 0$, then we have
	\begin{equation}\label{eq:Mlimit}
		M_x(u; 0+) = \lim_{r \to 0} \frac{|\{x \in B_r(x) : \chi(x) = 1\}|}{r^n} \in [0, |B_1|]
	\end{equation}
The limit on the right existing is part of the conclusion here.
Moreover, the function $x\mapsto M_x(u; 0+)$ is upper semicontinuous on $\W$.
\end{proposition}
\begin{proof}
	We assume $x = 0$. That $M_0(r) \rightarrow - \infty$ if $u(0) > 0$ is clear: in this case, $D_0(r)$ remains bounded while $r^2 H_0(r) \rightarrow u^2(0) |\partial B_1|$, so $-H_0(r) \rightarrow - \infty$. We will therefore assume $u(0) = 0$ and show \eqref{eq:Mlimit}. 
	
	First, in this case $M_0(r)$ is bounded in terms of the Lipschitz constant $C_V$ of $u$: $D_0(r) \leq |B_1|(C_V^2 + 1)$, while $H_0(r) \leq C_V^2 |\partial B_1|$. In particular, $M_0(0+)$ is finite.

	Consider $u_r(y) = \frac{u(r y)}{r}$, $\chi_r(y) = \chi(r y)$; these functions also form a variational solution and in particular $u_r$ is Lipschitz uniformly in $r$. For each sequence $r_k \to 0$, we can find a subsequence (still denoted $r_k$) along which $u_{r_k} \rightarrow u_0$ locally uniformly and weakly in $W^{1, 2}_{\mathrm{loc}}$ to some Lipschitz function $u_0 : \R^n \rightarrow [0, \infty)$. We can improve this convergence on $A = \{u_0 > 0\}$: for each $y \in A$ we also have $u_{r_k} > 0$ on some small ball $B_\e(y) \ss A$ for all $k$ large enough, so from elliptic estimates we have $\n u_{r_k} (y)\rightarrow \n u_0(y)$. In particular, we have that $u_0$ is harmonic on $A$.

	Take any $\eta \in C^\infty_c$ with $\eta \geq 0$, and a sequence $\e \to 0+$ such that $\{u_0 > \e \}$ is smooth on a neighborhood of $\supp \eta$. Then
	\[
		\int |\n u_0|^2 \eta = \lim_{\e \to 0+} \n u_0 \cdot \n (u_0 - \e)_+ \eta = - \lim_{\e \to 0+} \int \n \eta \cdot \n u_0 (u_0 - \e)_+ = - \int u_0 \n \eta \cdot \n u_0,
	\]
	where the first and last steps used dominated convergence and the middle step that $u_0$ is harmonic on $\{u_0 > \e \} \ss A$. The same computation is valid with $u_{r_k}$ in place of $u_0$, and so we see that
	\[
		\int |\n u_0|^2 \eta = - \int u_0 \n \eta \cdot \n u_0 = \lim_{r_k \to 0} - \int u_{r_k} \n \eta \cdot \n u_{r_k} = \lim_{r_k \to 0}  \int |\n u_{r_k}|^2 \eta.
	\]
	This implies that $u_{r_k} \rightarrow u_0$ strongly in $W^{1, 2}_{\mathrm{loc}}$ as $k\to\infty$.

    Moreover, by Proposition \ref{prop:monotone}, $u_0$ is a one-homogeneous function, i.e. $u_0(s y) = s u_0(y)$.

	From the strong convergence of $u_{r_k}$ in $W^{1, 2}(B_t)$, as well as uniform convergence on a neighborhood of $B_t$, we have that
	\[
		\frac{1}{t^n}\int_{B_t} |\n u_0|^2 - \frac{1}{t^{n + 1}} \int_{\partial B_t} u_0^2 \> d\mathcal{H}^{n-1}= \lim_{r_k \to 0} \frac{1}{t^n}\int_{B_t} |\n u_{r_k}|^2 - \frac{1}{t^{n + 1}} \int_{\partial B_t} u_{r_k}^2 \> d\mathcal{H}^{n-1}.
	\]
	The left-hand side here is zero, using only that $u_0$ is harmonic on $A$ and one-homogeneous: indeed, arguing as in \eqref{eq:intbyparts} we have that
	\[
		\frac{1}{t^n}\int_{B_t} |\n u_0|^2 - \frac{1}{t^{n + 1}} \int_{\partial B_t} u_0^2 \> d\mathcal{H}^{n-1}= \frac{1}{t^{n + 1}} \int_{\partial B_t} \left(u_0 \n u_0 \cdot y - u_0^2 \right)\> d\mathcal{H}^{n-1}= 0.
	\]
	It follows that
	\[
		0 = \lim_{r_k \to 0} \frac{1}{t^n}\int_{B_t} |\n u_{r_k}|^2 - \frac{1}{t^{n + 1}} \int_{\partial B_t} u_{r_k}^2\> d\mathcal{H}^{n-1} = \lim_{r_k \to 0} M_0(u; r_k) -  \frac{|\{x \in B_{r_k}(x) : \chi(x) = 1\}|}{r_k^n},
	\]
	so
	\[
		M_0(u; 0+) = \lim_{r_k \to 0}\frac{|\{x \in B_{r_k}(x) : \chi(x) = 1\}|}{r_k^n}.
	\]
	As this was true along every sequence, we obtain the conclusion.

The upper semicontinuity of $M_x(u; 0+)$ is a direct consequence of it being a monotone limit of continuous functions. Indeed, let $\delta>0$ and $K<+\infty$. Then Proposition
\ref{prop:monotone} implies that
\[ M_x(u; 0+)\leq M_x(u; r) \leq M_{x_0}(u; r) + \frac{\delta}{2}
\leq \left\{\begin{array}{ll}
M_{x_0}(u; 0+)+\delta, &\text{if }M_{x_0}(u; 0+)>-\infty,\\
-K, &\text{if }M_{x_0}(u; 0+)=-\infty,\end{array}\right.\]
if we choose for fixed $x_0$ first $r>0$ and then $|x-x_0|$ small enough.
\end{proof}

\section{The frequency formula}\label{sec:freq}

Given a variational solution $(u, \chi)$ on $\W$ define the \emph{frequency formula} to be
\[
	N_x(u; r) = \frac{D_x(u; r) - |B_1|}{H_x(u; r)}
\]
for any $B_r(x) \ss \W$ for which the denominator is positive. As $u$ is subharmonic (from Proposition \ref{prop:subharmonic}), $H_x(r) = 0$ implies that $u \equiv 0$ on $B_r(x)$, so $D_x(r), M_x(r)$ are both zero in this case. With this in mind, we further define $N_x(u; r) = - \infty$ when $H_x(r) = 0$.

Also helpful to simplify notation will be the \emph{volume difference}
\[
	V_x(u; r) = \frac{\int_{B_r(x)} \left(1 - \chi\right)}{H_x(u; r)} \geq 0.
\]
With these definitions, we have
\begin{equation}\label{eq:freqplusvol}
	N_x(u; r) + V_x(u; r) = \frac{r \int_{B_r(x)}|\n u|^2}{\int_{\partial B_r}u^2\> d\mathcal{H}^{n-1}},
\end{equation}
the more familiar Almgren frequency of a harmonic function.

\begin{lemma}[cf. {\cite[Theorem 7.1]{VW}} which contains the frequency formula for a related problem]\label{lem:frequency}
	Let $(u, \chi)$ be a variational solution on $\W$, and $B_r(x) \ss \W$. Assume that $x\in \overline{\{u>0\}}$ and that $M_x(u; t) \geq |B_1|$ for some $t < r$.
 Then on the interval $(t, r)$the frequency formula satisfies the bound $N_x(u; s) \geq 1$, is absolutely continuous as well as nondecreasing, and satisfies
	\begin{align}
		\frac{d}{ds}N_x(u; s) &= \frac{2}{H_x(u; s) s^{n + 2}} \int_{\partial B_s(x)} \left(\n u \cdot y  - [N_x(u; s) + V_x(u; s)] u \right)^2\> d\cH^{n-1}(y) \notag\\
		 &+ \frac{2}{s}\left[V^2_x(u; s) + V_x(u; s)(N_x(u; s) - 1)\right].\label{eq:nprime}
	\end{align}
\end{lemma}

\begin{remark}
	It is possible to rewrite the formula \eqref{eq:nprime} as
	\begin{equation} \label{eq:nprime2}
		\frac{d}{dr}N_x(u; s) = \frac{2}{H_x(u; s) s^{n + 2}} \int_{\partial B_s(x)} \left(\n u \cdot y  - N_x(u; s) u \right)^2\> d\cH^{n-1}(y) + \frac{2}{s}V_x(u; s)(N_x(u; s) - 1)
	\end{equation}
	by directly multiplying through and using \eqref{eq:intbyparts} and \eqref{eq:freqplusvol}. 
This equivalent version is more helpful when using the first term on the right, while \eqref{eq:nprime} is more helpful when we focus on
the $V^2$-term or need more control on the coefficient relevant to the order of homogeneity of $u$.
\end{remark}

\begin{proof}
	Without loss of generality take $x = 0$. First, by Proposition \ref{prop:monotone}, $M_0(s) \geq M_0(t) \geq |B_1|$, so $N_0(s) \geq 1 > - \infty$. That proposition also gives that $N_0(s)$ is absolutely continuous, and lets us compute (using $H = H_0(u; s)$, $H' = \frac{d}{dr}H_0(u; s))$ and similarly for the other quantities):
	\begin{align}\label{eq:freqi1}
		N' &= \frac{D'}{H} - \frac{H'(D - |B_1|)}{H^2} \notag \\
		& = \frac{2}{H s^{n + 2}}\int_{\partial B_s} \left((\n u \cdot y)^2 - u \n u \cdot y \right)\> d\cH^{n-1} - 2 N \left[\frac{1}{H s^{n + 2}}\int_{\partial B_s} u \n u \cdot y \> d\cH^{n-1} - \frac{1}{s}\right].
	\end{align}
	One of the terms with $N$ can be rewritten with the help of \eqref{eq:intbyparts}, giving
	\begin{align*}
		- 2 N \frac{1}{H s^{n + 2}}\int_{\partial B_s} u \n u \cdot y \> d\cH^{n-1} &= - \frac{2}{H s^{n+2}} \frac{1}{H s^{n + 1}}\left(\int_{\partial B_s} u \n u \cdot y \> d\cH^{n-1}\right)^2 \\
		&+ 2 V \frac{1}{H s^{n + 1}} \int_{B_s} |\n u|^2.
	\end{align*}
	Grouping the first term of this with the first term on the right of \eqref{eq:freqi1},
	\begin{align*}
		\frac{2}{H s^{n + 2}}&\left[\int_{\partial B_s} (\n u \cdot y)^2 \> d\cH^{n-1} - \frac{1}{H s^{n + 1}}\left(\int_{\partial B_s} u \n u \cdot y \> d\cH^{n-1}\right)^2\right] \\
		& = \frac{2}{H s^{n + 2}}\left[\int_{\partial B_s} \left(\n u \cdot y - u \frac{1}{H s^{n + 1}} \int_{\partial B_s} u \n u \cdot y \> d\cH^{n-1} \right)^2 \> d\cH^{n-1}\right] \\
		& = \frac{2}{H s^{n + 2}}\left[\int_{\partial B_s} \left(\n u \cdot y - u [N + V] \right)^2 \> d\cH^{n-1}\right].
	\end{align*}
	The first step is a direct computation using the definition of $H$, while the second used \eqref{eq:intbyparts} and \eqref{eq:freqplusvol}.

	The remaining terms of \eqref{eq:freqi1} may be rearranged as follows:
	\begin{align*}
		- \frac{2}{H s^{n + 2}}&\int_{\partial B_s} u \n u \cdot y \> d\cH^{n-1} + 2 V \frac{1}{H s^{n + 1}} \int_{B_s} |\n u|^2 + \frac{2N}{s} \\
		&= \frac{2}{H s^{n + 1}}\int_{B_s}|\n u|^2 [V - 1] + \frac{2 N}{s} \\
		&= \frac{2}{s}[(N + V)(V -1) + N] \\
		&= \frac{2}{s}[V^2 + V(N - 1)],
	\end{align*}
	where the first step used \eqref{eq:intbyparts} and the second used \eqref{eq:freqplusvol}. This gives the formula in \eqref{eq:nprime}.

	Finally, to see that $N$ is nondecreasing, observe that $V \geq 0$, while as we previously observed $N \geq 1$; this ensures that all terms in the expression for $N'$ are nonnegative.
\end{proof}

\begin{lemma}\label{lem:hdoubling}
	Let $(u, \chi)$ be a variational solution on $\W$, and $B_r(x) \ss \W$. Assume that 
$x\in \overline{\{u>0\}}$, $M_x(u; r/2) \geq |B_1|$ and that
$N_x(u; r) \leq N_+$ for some $N_+ > 1$. Then there is a constant $C = C(N_+, n)$ such that for any $s \in [\frac{r}{2}, r]$,
	\[
		H_x(u; s) \leq H_x(u; r) \leq C H_x(u; s).
	\]
\end{lemma}

\begin{proof}
	After scaling we may assume that $x = 0$ and $r = 1$. From the assumptions $N_0(s) \geq 1$, and as $V_0(s) \geq 0$ we have $N_0(s) + V_0(s) \geq 1$. From Proposition \ref{prop:monotone},
	\begin{align*}
		\frac{d}{ds} H_0(s) &= \frac{2}{s}\left[\frac{1}{s^{n+1}}\int_{\partial B_s} u \n u \cdot yd \cH^{n-1}(y) - H_0(s)\right] \\
		&= \frac{2}{s} H_0(s) (N_0(s) + V_0(s) - 1)
		\quad \geq 0.
	\end{align*}
	The second line used \eqref{eq:intbyparts} and \eqref{eq:freqplusvol}. This gives the lower bound of our lemma.

	For the upper bound, we can argue directly using Lemma \ref{lem:frequency}:
	\[
		\int_{\frac{1}{2}}^1 \frac{2 V_0^2(s)}{s} ds \leq \int_{\frac{1}{2}}^1 \frac{d}{ds} N_0(s)  ds = N_0(1) - N_0(\frac{1}{2}) \leq N_+ - 1,
	\]
	so
	\[
		\int_{\frac{1}{2}}^1 V_0 \> ds \leq \sqrt{\int_{\frac{1}{2}}^1 V_0^2(s) \> ds } \leq \sqrt{\int_{\frac{1}{2}}^1 \frac{2 V_0^2(s)}{s} \> ds } \leq \sqrt{N_+ - 1}.
	\]
	Integrating the formula
	\[
		\frac{d}{ds}\log H_0(s) = \frac{2}{s} [N_0(s) + V_0(s) - 1] \leq 4 [N_+ - 1 + V_0(s)]
	\]
	over $[s, 1]$ and using that $s \geq \frac{1}{2}$ then leads to
	\[
		H_0(1) \leq H_0(s) e^{2(N_+ - 1) + 4\sqrt{N_+ - 1}}.
	\]
\end{proof}

\begin{lemma}\label{lem:vdoubling}
	Let $(u, \chi)$ be a variational solution on $\W$, and $B_r(x) \ss \W$. Assume that 
$x\in \overline{\{u>0\}}$, that
$M_x(u; r/2) \geq |B_1|$ and that
$N_x(u; r) \leq N_+$ for some $N_+ > 1$. Then there is a constant $C = C(N_+, n)$ such that
	\[
		\sup_{s \in [\frac{r}{2}, r]} V_x(u; s) \leq C V_x(u; r).
	\]
\end{lemma}

\begin{proof}
	As usual, it suffices to consider $x = 0$ and $r = 1$. Note that directly from the definitions,
	\[
		V_0(s) = \frac{1}{H_0(s)}\int_{B_s}(1 - \chi) \leq \frac{1}{H_0(s)}\int_{B_1}(1 - \chi) = \frac{H_0(1)}{H_0(s)} V_0(1).
	\]
	Applying Lemma \ref{lem:hdoubling} gives $\frac{H_0(1)}{H_0(s)} \leq C$.
\end{proof}

\begin{lemma}\label{lem:vvanishing}
	Let $(u, \chi)$ be a variational solution on $B_r(x)$ and $t \in (0, \frac{r}{2}]$. Assume that 
$x\in \overline{\{u>0\}}$, that
$M_x(u; t) \geq |B_1|$, $N_x(u; r) - N_x(u; t) < \d$, and $N_x(u; r) < N_+$. Then there exists a constant $C = C(N_+, n)$ such that
	\[
		V_x(u; t) \leq C \sqrt{\d}.
	\]
\end{lemma}

This lemma essentially asserts that for $t \leq r/2$, we have $V_x(t)$ bounded by a constant depending only on $n$ and $N_+$, and moreover $V_x(t) \rightarrow 0$ if $N_x(u; 2 t) - N_x(u; t) \rightarrow 0$ (uniformly over $u$ with bounded $N_+$). In particular, this implies that at any point $x$ where $N_x(u, 0 + )\geq 1$, we have $V_x(r) \rightarrow 0$.

\begin{proof}
	By a straightforward scaling and translation, it suffices to show this with $x = 0$ and $t = 1$, with $r \geq 2$. We may, without loss of generality, replace $r$ by $2$, as the monotonicity of $N_0$ from Lemma \ref{lem:frequency} gives us $N_0(2) \leq N_0(r) < N_+$ and $N_0(2) - N_0(1)\leq N_0(r) - N_0(1) < \d$.

	From \eqref{eq:nprime}, we have that
	\[
		\int_1^2 \frac{2 V^2_0(s)}{s} \> ds \leq N_0(2) - N_0(1) < \d,
	\]
	so there must be at least one $s \in [1, 2]$ with $V_0(s) \leq \sqrt{\d}$. Applying Lemma \ref{lem:vdoubling},
	\[
		V_0(1) \leq C(N_+)\sqrt{\d}.
	\]
\end{proof}

\begin{lemma}\label{lem:changeofpoint}
	Let $(u, \chi)$ be a variational solution on $B_r(x)$ with $x\in \overline{\{u>0\}}$,
$M_x(u; r/4) \geq |B_1|$ and $N_x(u; r) \leq N_+$.
Then there exists a $C = C(n)$ such that for any $y \in B_{r/4}(x)$ and $s < \frac{1}{2} r$,
	\[
		N_y(u; s) \leq \max\{1, C N_x(u; r)\}
	\]
	and
	\[
		H_y(u; \frac{5}{8}r) \leq C H_x(u; r).
	\]
\end{lemma}

\begin{proof}
	We show this with $x = 0$ and $r = 1$. First, from Lemma \ref{lem:hdoubling} we have that
	\[
		\int_{\partial B_1} u^2 \> d\mathcal{H}^{n-1}\leq C(N_+) \int_{\partial B_s} u^2\> d\mathcal{H}^{n-1}
	\]
	for $s \in [\frac{1}{4}, 1]$. Integrating this gives
	\[
		\int_{\partial B_1} u^2 \> d\mathcal{H}^{n-1}\leq C \int_{B_{3/8} \setminus B_{1/4}} u^2 \leq C \int_{B_{3/8}} u^2.
	\]

	For any subharmonic function $v$,
	\[
		0 \leq \frac{d}{dt} \frac{1}{t^n}\int_{B_t} v = \frac{1}{t^n}\int_{\partial B_t} v \> d\mathcal{H}^{n-1}- \frac{n}{t^{n+1}}\int_{B_t} v,
	\]
	and in particular this can be applied to translates of $u^2$. Take any $y \in B_{1/4}$; then
	\[
		\int_{B_{3/8}} u^2 \leq \int_{B_{5/8}(y)} u^2 \leq C \int_{\partial B_{5/8}(y)}u^2\> d\mathcal{H}^{n-1}.
	\]
	So this gives $C H_y(5/8) \geq H_0(1)$. On the other hand, we have the inequality $D_y(5/8) \leq (\frac{8}{5})^n D_0(1)$ directly from the definition, showing that
	\[
		N_y(5/8) \leq C \> N_0(1).
	\]

	Finally, observe that from Lemma \ref{lem:frequency} that if for any $s$ we have $N_y(s) > 1$, it remains nondecreasing (and hence $>1$ for all larger $s$), leading to $N_y(s) \leq N_y(5/8) \leq C\> N_0(1)$.
\end{proof}

\section{Analysis of frequency blow-ups}\label{sec:blowup}

Throughout this section we consider a sequence of variational solutions $(u_k, \chi_k)$ on $\W$, with $B_2\ss \W$, $0\in \overline{\{u_k>0\}}$,
$M_0(u; 1/2) \geq |B_1|$ and $N_0(u_k; 2) \leq N_+$ for some constant $N_+$, and $\d_k := N_0(u_k; 2) - N_0(u_k; 1/2) \rightarrow 0$. An example of such a sequence arises by considering blow-up sequences $\frac{u(r_k \cdot)}{r_k}$ of a given variational solution $(u, \chi)$ with $N_0(x; 0+) \geq 1$, but it will be useful to treat this more general case for use in compactness arguments below.

Given such a sequence, we consider the \emph{renormalized sequence}
\[
	v_k(y) = \frac{u_k(y)}{\sqrt{H_0(u_k; 1)}}.
\]
The $v_k$ have the property that $\int_{\partial B_1}v_k^2 = 1$.

\begin{lemma}\label{lem:renormalizedlimits}
	A subsequence of the renormalized sequence $v_k$ converges weakly in $W^{1,2}(B_1)$ and strongly in $L^2(B_1)\cap L^2(\partial B_1)$ to a function $v \in W^{1,2}(B_1)$. This function $v$ satisfies the following properties:
	\begin{enumerate}
		\item $v \geq 0$ on $B_1$
		\item $v$ is subharmonic on $B_1$
		\item On $B_1 \sm B_{1/2}$, $v$ is homogeneous of order $N_\infty := \lim_{k \rightarrow \infty} N_0(u_k; 1/2) \geq 1$
		\item $\int_{\partial B_1} v^2\> d\mathcal{H}^{n-1} = 1$.
	\end{enumerate}
\end{lemma}

\begin{proof}
	We may choose a subsequence so that $N_0(u_k; 1/2)$ converges to some $N_\infty \in [1, N_+]$. By Lemma \ref{lem:vvanishing}, we have that $V_0(u_k; 1) \rightarrow 0$, so
	\[
		\int_{B_1}|\n v_k|^2 = \frac{\int_{B_1} |\n u_k|^2}{H_0(u_k; 1)} = N_0(u_k; 1) + V_0(u_k; 1) \rightarrow N_\infty.
	\]
	Along with the fact that $\int_{\partial B_1}v_k^2 \> d\mathcal{H}^{n-1}= 1$, this means that the $v_k$ are uniformly bounded in $W^{1,2}(B_1)$, and so we may find a subsequence which converges to some $v \in W^{1,2}(B_1)$ weakly in $W^{1,2}(B_1)$ and strongly in $L^2(B_1)\cap L^2(\partial B_1)$
by Rellich's theorem as well as the compact embedding on the boundary. 
Consequently (4) holds.
As $v_k \geq 0$, it is clear that $v \geq 0$.

	Take any $\eta \in C_c^\infty(B_1)$ with $\eta \geq 0$. Then
	\[
		\int \n v \cdot \n \eta = \lim_{k \rightarrow \infty} \int \n v_k \cdot \n \eta \leq 0,
	\]
	as the $u_k$ (and hence the $v_k$) are subharmonic from Proposition \ref{prop:subharmonic}. Therefore, $v$ is also subharmonic.

	Applying \eqref{eq:nprime2} to $u_k$ and integrating from $1/2$ to $1$ leads to
	\[
		\int_{B_1 \sm B_{1/2}} \frac{1}{H_0(u_k; |y|)} |\n u_k(y) \cdot y - N_0(u_k; |y|) u_k(y)|^2 \> dy \leq C \d_k \rightarrow 0.
	\]
	As $N_0(u_k; |y|) \rightarrow N_\infty$ uniformly in $|y|$, while $H_0(u_k; |y|) \geq c H_0(u_k; 1)$ by Lemma \ref{lem:hdoubling},
	\[
		\int_{B_1 \sm B_{1/2}}  |\n v_k(y) \cdot y - N_\infty v_k(y)|^2 \> dy \rightarrow 0.
	\]
	The functional $w \mapsto \int_{B_1 \sm B_{1/2}}|\n w(y) \cdot y - N_\infty w(y)|^2$ is convex, and so lower semicontinuous under weak convergence in $W^{1,2}$. This leads to
	\[
		\int_{B_1 \sm B_{1/2}} |\n v(y) \cdot y - N_\infty v(y)|^2 \> dy \leq \lim_{k \rightarrow \infty} \int_{B_1 \sm B_{1/2}}  |\n v_k(y) \cdot y - N_\infty v_k(y)|^2 = 0,
	\]
	and so $\n v(y) \cdot y - N_\infty v(y)$ almost everywhere on $B_1 \sm B_{1/2}$. This is equivalent to order $N_\infty$ homogeneity, i.e. $v(y) = |y|^{N_\infty} v(y/|y|)$ almost everywhere, where $v(y/|y|)$ is the trace of $v$ on $\partial B_1$.
\end{proof}

In the above argument, it is important to note that it is not clear that $v_k \rightarrow v$ \emph{strongly} in $W^{1,2}(B_1)$ (or the annular region $B_1\sm B_{1/2}$). The limit of $\int_{B_1} |\n v_k|^2$ is $N_\infty$, as we computed, but it may be possible that this is strictly larger than $\int_{B_1} |\n v|^2$. If $v$ were harmonic, $\int_{B_1}|\n v|^2 = N_\infty \int_{\partial B_1} v^2\> d\mathcal{H}^{n-1}$ would be elementary to verify from the homogeneity of $v$, but notice that we have not passed any PDE for $v_k$ to the limit: this appears to also require an improvement of convergence.

\begin{conjecture}
	The renormalized sequence $v_k \rightarrow v$ strongly in $W^{1,2}(B_1 \sm B_{1/2})$.
\end{conjecture}

Fortunately we do not need to solve this highly nontrivial concentration-compactness
conjecture in full strength. Instead, we only consider the case when the blow-up is one-dimensional.

\begin{theorem}\label{thm:approx1Dlimits}
	Write $x = (x', x_n) \in \R^{n - 1} \times \R$ and set $A = B_1 \sm \overline{B_{1/2}}$. Assume that $n \geq 2$ and that the normalized sequence $(v_k)_{k\in \N}$ defined above satisfies
	\[
		\sum_{i = 1}^{n-1} \int_A |\partial_i v_k|^2 \rightarrow 0 \text{ as }k\to\infty.
	\]
	Then on $A$, $v$ is independent of the first $n - 1$ coordinates, i.e. $v(x) = h(x_n)$, $N_\infty = 1$, and $h(x_n) = \a(n) |x_n|$.
\end{theorem}

Note that the constant $\a(n)$ is explicit and fully determined by the normalization $$\int_{\partial B_1} v^2\> d\mathcal{H}^{n-1} = 1.$$ This theorem is still valid when $n = 1$ (as can be seen by simply classifying all variational solutions), but the proof below uses that $A$ is connected.

\begin{proof}
	That $v$ is independent of the first $n - 1$ coordinates is immediate from the assumption. The function $h$ must be of the form $h(x_n) = \b_+ (x_n)_+^{N_\infty} + \b_- (x_-)_+^{N_\infty}$ from the homogeneity of $v$, and at least one of $\b_\pm \neq 0$.

	We also recall that 
	\[
		\int_{A} |\n v(y) \cdot y - N_\infty v(y)|^2 \> dy = 0=\lim_{k \rightarrow \infty} \int_{A} |\n v_k(y) \cdot y - N_\infty v_k(y)|^2 \> dy.
	\]
	By expanding the square and using the weak convergence of $v_k \rightarrow v$ in $W^{1, 2}$, this leads to
	\[
		\lim_{k \rightarrow \infty} \int_A |\p_n v_k(y) y_n|^2 \> dy = \int_A |h'(y_n)|^2 y_n^2 \> dy.
	\]
	In particular, take any $U \cc A \sm \{x_n = 0\}$; we have shown that $v_k \rightarrow v$ strongly in $W^{1, 2}(U)$ as $k\to\infty$,
as weak convergence combined with convergence of the weighted $L^2$-norm
implies strong convergence. Taking any $\xi \in C_c^{0, 1}(U; \R^n)$, this implies that
	\[
		\int \left(|\n v|^2 \dvg \xi - 2 \n v \cdot D\xi \n v\right) = 0
	\]
	by passing the domain variation formula for $u_k$ to the limit and using that $V_0(u_k; 1)\rightarrow 0$.
Due to the one-dimensional structure of $v$, this may be rewritten (integrating by parts against constants in the $x'$ variable) as
	\begin{equation}\label{eq:1Dvariations}
		0 = \int (h')^2 \p_n \xi_n = -\int 2 h' h'' \xi_n.
	\end{equation}
	This implies that either $\b_\pm = 0$ or $h'' = 0$ on each component of $A \sm \{x_n = 0\}$. On at least one of the components, then, $\b_\pm $ is nonzero and so we have shown that $N_\infty = 1$.

	Next, we claim that $v_k \rightarrow v$ strongly on the whole set $A' = B_{7/8} \sm B_{5/8}$.
Take any nonnegative $\eta \in C_c^\infty(A)$, and integrate by parts, using that we already know that $v \Delta v = 0$:
	\[
		\int |\n v|^2 \eta = - \int v \n \eta \cdot \n v - \int v \eta \Delta v = - \int v \n \eta \cdot \n v.
	\]
	Then this passes to the limit under weak convergence, so
	\[
		\int |\n v|^2 \eta = - \lim_{k \rightarrow \infty} \int v_k \n \eta \cdot \n v_k = \lim_{k \rightarrow \infty} \int |\n v_k|^2 \eta,
	\]
	where the last integration by parts used that $\Delta v_k = 0$ on $\{v_k > 0\}$. As above,
this gives strong convergence locally in $A'$. As a consequence, \eqref{eq:1Dvariations} is valid for any $\xi \in C_c^1(A')$.

	As $n \geq 2$, we may find an $x'$ such that $(x', t) \in A'$ for $t \in (-t_0, t_0)$ for some small $t_0$. Setting $\xi_n \geq 0$ to be supported near $(x', 0)$ and symmetric with respect to $x_n$, \eqref{eq:1Dvariations} becomes
	\[
		[\b_+^2 - \b_-^2]\int_{\{x_n=0\}} \xi_n \> d\mathcal{H}^{n-1}= 0,
	\]
	so $\b_+ = \b_-$.
\end{proof}

\section{Core estimates}\label{sec:estimate}

For $(u, \chi)$ a variational solution on $\W$, we define the \emph{set of highest density}
\[
	\Sigma^H = \{ x \in \W : x\in \overline{\{u>0\}}\text{ and }M_x(u; 0+) \geq |B_1| \}.
\]
In Lemma \ref{lem:density0char} we will show that every degenerate free boundary point
is contained in $\Sigma^H$.

Recall from Proposition \ref{prop:Mlimits} that $|B_1|$ is the largest possible value $M_x(0+)$ may attain at $x$, and from Lemma \ref{lem:frequency}---using the monotonicity of
$M_x$---that at any $x\in \Sigma^H$ the frequency function $N_x(r)$ is non-decreasing and
satisfies $N_x(u; 0+) \geq 1$.

It will also be useful to define
\[
	\Sigma^H_r = \{ x \in \W : B_r(x)\cc \W,  x\in \overline{\{u>0\}}\text{ and }
M_x(u; r) \geq |B_1| \}.
\]
The relationship between these is that for $s < r$, $\Sigma^H_s \cap \W_r \ss \Sigma^H_r$, where $\W_r = \{ x \in \W : B_r(x) \cc \W \}$ is the set of points a distance more than $r$ from the boundary of $\W$. Similarly, $\Sigma^H \cap \W_r \ss \Sigma^H_r$, while $\bigcap_{0 < s < r} \Sigma^H_s = \Sigma^H \cap \W_r$. The sets $\Sigma^H_r$ are relatively closed in $\W_r$, while $\Sigma^H$ is relatively closed in $\W$. 

\subsection{\texorpdfstring{$L^2$}{L2}-Subspace Approximation}

Let $\mu$ be a finite Borel measure, and set
\[
	\b_\mu^2(x, r) = \inf_{L} \int_{B_r(x)} \frac{d^2(y, L)}{r^2} \frac{d\mu}{r^{n-1}}
\]
to be the \emph{Jones number} of $\mu$, where the infimum is taken over all affine hyperplanes $L \ss \R^n$. Smallness of $\beta$ measures how concentrated on a hyperplane $\mu$ is. The first of two estimates we prove in this section is the following theorem, which controls this quantity by the drop in frequency on the support of $\mu$:

\begin{theorem}[$L^2$-Subspace Approximation]\label{thm:l2subspaceapprox}
	Let $(u, \chi)$ be a variational solution on $B_{8r}(x)$ and $\mu$ a finite (positive) Borel measure supported on $\Sigma^H_r$. Assume that $x \in \Sigma^H_r$ and that $N_x(u; 50 r) \leq N_+$. Then there is a constant $C = C(N_+, n)$ such that
	\[
		\b_\mu^2(x, r) \leq \frac{C}{r^{n-1}} \int_{B_r(x)} [N_y(u; 20 r) - N_y(u; r)] \> d\mu(y).
	\]
\end{theorem}

This theorem is perhaps best understood as an estimate on discrete measures: the left-hand side is large if the measure's support is in general position, i.e. at least one point is far from any hyperplane. Then this estimate guarantees that, at least at one point $y$ in the support, $N_y(2r) - N_y(r)$ has to be large. Indeed, if the frequency was almost constant at a collection of points in general position, then $u$ would have to be approximately homogeneous with respect to each one of them, and this is impossible for degrees of homogeneity other than $0$. The proof is based on a quantitative version of this argument.

\begin{lemma}\label{lem:twopointhom}
	Let $u$ be a Lipschitz function on $\R^n$, and fix a point $x \in B_1\sm \{0\}$ and number $N_+ > 1$. Assume there are two constants $N_0, N_x \in [1, N_+]$ and a number $\d$ such that
	\[
		\int_{B_{10} \sm B_1} |\n u(y) \cdot y - N_0 u(y)|^2 \> dy \leq \d
	\]
	and also
	\[
		\int_{B_{10} \sm B_1} |\n u(x + y) \cdot y - N_x u(x + y)|^2 \> dy \leq \d.
	\]
	Then there is a constant $C = C(N_+, n)$ such that
	\begin{equation}\label{eq:twopointhomc1}
		|N_0 - N_x|^2 \int_{B_{10} \sm B_1} u^2 \leq C \d
	\end{equation}
	and
	\begin{equation}\label{eq:twopointhomc2}
		\int_{B_{9} \sm B_2} |\n u(y) \cdot x |^2dy  \leq C \d.
	\end{equation}
\end{lemma}

\begin{proof}
	We may write $x = y + (x - y)$ for any $y$, to obtain that $\n u(y) \cdot x = \n u(y) \cdot y + \n u(x + (y - x)) \cdot (x - y)$. Integrating this,
	\begin{align*}
		\|\n u(y)& \cdot x - (N_0 - N_x) u(y) \|_{L^2(B_{9} \sm B_{2}, dy)} \\
		&\leq \|\n u(y) \cdot y - N_0 u(y)\|_{L^2(B_{9} \sm B_{2}, dy)} \\&\quad + \|\n u(x + (y - x)) \cdot (y - x) - N_x u(y)\|_{L^2(B_{9} \sm B_{2}, dy)} 
		\leq 2 \sqrt{\d}.
	\end{align*}
	In particular, this means that \eqref{eq:twopointhomc1} implies \eqref{eq:twopointhomc2}.

	For any $y \in B_8 \sm B_3$, the line segment $y + t x, t\in [0, 1]$ is contained in $B_9 \sm B_2$. It follows that
	\begin{align*}
		u(y + x)e^{N_x - N_0} - u(y) &= \int_0^1 \frac{d}{dt} \left(e^{(N_x - N_0)t}u(y + t x)\right) \> dt \\
		&= \int_0^1 e^{(N_x - N_0)t}[\n u(y + t x) \cdot x - (N_0 - N_x) u(y + t x)]dt.
	\end{align*}
	We square and integrate in $y$ over the annulus $B_8 \sm B_3$:
	\begin{align}
		\int_{B_8 \sm B_3} |&u(y + x)e^{N_x - N_0} - u(y)|^2 \> dy \notag \\&\leq C(N_+)\int_0^1 \int_{B_8 \sm B_3}|\n u(y + t x) \cdot x - (N_0 - N_x) u(y + t x)|^2 \> dy \> dt \notag \\
		& \leq C \int_{B_9 \sm B_2}|\n u(y) \cdot x - (N_0 - N_x) u(y)|^2 \> dy \notag \\
		& \leq C \d. \label{eq:twopointhomi1}
	\end{align}

	In a somewhat similar fashion, consider the segment $t y$, with $y \in \partial B_s$ and $t \in [1, T]$ such that $s T \leq 10$. In this case
	\[
		T^{- N_0} u(T y) -  u(y) = \int_1^T \frac{d}{dt} t^{- N_0} u(t y) \> dt = \int_1^T t^{-N_0 - 1}[ \n u(t y) \cdot (t y) - N_0 u (t y)] \> dt.
	\]
	Squaring and integrating this over the sphere $\p B_s$,
	\begin{equation}\label{eq:twopointhomi3}
		\int_{\p B_s}|T^{- N_0} u (T y) - u(y)|^2 \> d\mathcal{H}^{n-1}(y) \leq C(N_+) \d,
	\end{equation}
where the constant $C(N_+)$ depends on $T$ and $s$ but can be chosen
uniformly in $s \in [1, 10]$.
	This allows us to compare the integrals of $u^2$ between spheres of different sizes:
	\begin{align*}
		\int_{\p B_s} T^{-2N_0} u^2(T y) \> d\mathcal{H}^{n-1}(y) &\leq 2 \int_{\p B_s} \left(u^2 + |T^{- N_0} u (T y) - u(y)|^2\right)\> d\mathcal{H}^{n-1}(y) \\& \leq C\int_{\p B_s} u^2\> d\mathcal{H}^{n-1} + C \d,
	\end{align*}
	so
	\[
		\int_{\p B_{Ts}} u^2\> d\mathcal{H}^{n-1} \leq C \int_{\p B_{s}} u^2\> d\mathcal{H}^{n-1} + C \d.
	\]
	Similarly
	\[
		\int_{\p B_s} u^2\> d\mathcal{H}^{n-1} \leq C \int_{\p B_{Ts}} u^2\> d\mathcal{H}^{n-1} + C \d,
	\]
	and so averaging over annuli,
	\begin{equation}\label{eq:twopointhomi5}
		\frac{1}{C}\int_{B_{10}\sm B_1} u^2 - \d \leq  \int_{\p B_s} u^2\> d\mathcal{H}^{n-1} \leq C\int_{B_{10}\sm B_1} u^2 + C \d
	\end{equation}
	for every $s \in [1, 10]$.
	
	We may instead simply integrate \eqref{eq:twopointhomi3} over $[3,4]$ and choose $T = 2$, which leads to
	\begin{equation}\label{eq:twopointhomi2}
		\int_{B_4\sm B_3}|2^{- N_0} u (2 y) - u(y)|^2 \leq C(N_+) \d.
	\end{equation}
	An analogous computation centered around $x$ gives
	\begin{equation}\label{eq:twopointhomi4}
		\int_{B_4\sm B_3}|2^{- N_x} u (x + 2 y) - u(x + y)|^2 \leq C \d
	\end{equation}
	as well. On the other hand, from \eqref{eq:twopointhomi1} we know that (setting $A = e^{N_x - N_0}$, which is bounded above and below in terms of $N_+$)
	\[
		\int_{B_4 \sm B_3} |u(y) - A u(x + y)|^2 \leq C \d
	\]
	and
	\[
		\int_{B_4 \sm B_3} |u(2y) - A u(x + 2y)|^2 = 2^{-n}\int_{B_8 \sm B_6} |u(y) - A u(x + y)|^2 \leq C \d.
	\]
	Using the triangle inequality (with all norms being $L^2(B_{4}\sm B_3, dy)$),
	\begin{align*}
		\|2^{-N_x}&u(2y) - 2^{-N_0}u(2y)\| \leq \|2^{-N_0}u(2y) - u(y)\| + \|u(y) - A u(x + y)\| \\
		&+ \|A u(x + y) - A2^{-N_x} u(x + 2y)\| + \|A 2^{-N_x} u(x + 2y) - 2^{-N_x}u(2y)\| \\
		& \leq C \sqrt{\d},
	\end{align*}
	where we used \eqref{eq:twopointhomi2} and \eqref{eq:twopointhomi4} to bound the first and third terms respectively. Rewriting,
	\[
		|2^{-N_x} - 2^{-N_0}|^2 \int_{B_8 \sm B_6} u^2 \leq C \d.
	\]
	As $|2^{-N_x} - 2^{-N_0}| \geq c(N_+) |N_x - N_0|$, using \eqref{eq:twopointhomi5} this implies that
	\[
		|N_x - N_0|^2 \int_{B_{10} \sm B_1} u^2 \leq C \d,
	\]
	proving \eqref{eq:twopointhomc1}.
\end{proof}

\begin{proof}[Proof of Theorem \ref{thm:l2subspaceapprox}]
	After a one-homogeneous rescaling, we may assume that $r = 1$ and $x = 0$. Furthermore, we may as well assume that $\mu$ is supported on $B_1$ and normalized so that $\mu(B_1) = 1$ (the conclusion is homogeneous in $\mu$ and both sides are integrated over $B_1$).

	We apply Lemma \ref{lem:changeofpoint} to see that for any $x \in \Sigma^H_1 \cap \bar{B}_1$, $N_x(20) \leq C N_+$, and we also have that $\frac{1}{C} H_0(50) \leq H_x(s) \leq C H_0(50)$ from the second conclusion there (combined with Lemma \ref{lem:hdoubling} used repeatedly).
	
	Fix a point $x_0 \in \Sigma^H_1 \cap \bar{B}_1$ such that
	\[
		N_{x_0}(20) - N_{x_0}(1) = \min_{x \in \Sigma^H_1 \cap \bar{B}_1}\left(N_{x}(20) - N_{x}(1)\right).
	\]
	The quantity $N_{x}(r)$ is continuous in $x$ and $\Sigma^H_1\cap \bar{B}_1$ is closed, so the minimum here is attained. We also may as well assume that $N_{x_0}(20) - N_{x_0}(1)\leq \d$ for a small $\d = \d(N_+, n)$ (to be selected below): if not, the right-hand side in the conclusion is bounded from below, the left-hand side is bounded above by $1$, and so the theorem is trivial.
	
	At any $x \in \Sigma^H_1 \cap \bar{B}_1$, integrating \eqref{eq:nprime2} tells us that
	\[
		\int_1^{20} \int_{\partial B_s(x)}\frac{2}{H_x(s) s^{n+2}} |\n u(y) \cdot (y - x) - N_x(s) u(y)|^2\> d\mathcal{H}^{n-1}(y) \leq N_x(20) - N_x(1).
	\]
	By using that $H_x(s) \leq C H_0(50) \leq C H_{x_0}(10)$ and $0\leq N_x(s) - N_x(1) \leq N_{x}(20) - N_x(1)$,
	\[
		\int_{B_{20}(x) \sm B_1(x)} |\n u(y) \cdot (y - x) - N_x(1) u(y)|^2 \leq C H_{x_0}(10) [N_x(20) - N_x(1)].
	\]
	We insert this bound at $x$ and $x_0$ into Lemma \ref{lem:twopointhom}, so that \eqref{eq:twopointhomc2} gives
	\begin{equation}\label{eq:l2subspaceapproxi1}
		\int_{B_{18}(x_0) \sm B_4(x_0)} |\n u(y) \cdot (x - x_0)|^2 \> dy \leq C H_{x_0}(10) [N_x(20) - N_x(1)].
	\end{equation}
	Note how we used that $N_{x_0}(20) - N_{x_0}(1) \leq N_x(20) - N_x(1)$ here.

	Given an affine hyperplane $L = \{ (y - y_0) \cdot \nu = 0 \}$, the definition of $\b = \b_{\mu}(x, r)$ tells us that
	\[
		\b^2 \leq \int_{B_1} |(y - y_0) \cdot \nu |^2 \> d\mu(y).
	\]
	We plug into this formula $\nu = \frac{\n u(z)}{|\n u(z)|}$ for an arbitrary point $z\in\{u>0\}$ and $y_0 = x_0$ (the point selected above):
	\[
		\b^2 |\n u(z)|^2 \leq \int_{B_1} |\n u(z) \cdot (y - x_0)|^2 \> d\mu(y).
	\]
	Now we integrate over $A = \left(B_{10}(x_0) \sm B_4(x_0)\right)\cap \{u>0\}$ in the $z$ variable and interchange the order of integration, to get
	\[
		\b^2 \int_A |\n u(z)|^2 \> dz \leq \int_{B_1} \int_A |\n u(z) \cdot (y - x_0)|^2 \> dz \> d\mu(y).
	\]
	From \eqref{eq:l2subspaceapproxi1}, then,
	\[
		\b^2 \int_A |\n u(z)|^2 \> dz \leq CH_{x_0}(10) \int_{B_1} [N_y(20) - N_y(1)] \> d\mu(y).
	\]
	All that remains, then, is to show that
	\[
		\frac{1}{H_{x_0}(10)}\int_A |\n u(z)|^2 \> dz \geq c > 0.
	\]

	It is helpful here to use the monotonicity of $N_{x_0}(s)$:
	\[
		N_{x_0}(5) \leq N_{x_0}(10) \implies N_{x_0}(5) + V_{x_0}(5) \leq N_{x_0}(10) + V_{x_0}(10) + C \sqrt{\d},
	\]
	where we used Lemma \ref{lem:vvanishing} and $\d \geq N_{x_0}(20) - N_{x_0}(1)$ to bound the $V$ terms. This can be rewritten as
	\[
		\int_{B_{5}(x_0)} |\n u|^2 \leq \frac{1}{2^n}\frac{H_{x_0}(5)}{H_{x_0}(10)} \int_{B_{10}(x_0)} |\n u|^2 + C \sqrt{\d} H_{x_0}(5).
	\]
	From Lemma \ref{lem:hdoubling}, $H_{x_0}(5) \leq H_{x_0}(10)$, so
	\[
		\int_{B_{5}(x_0)} |\n u|^2 \leq \frac{1}{2^n} \int_{B_{10}(x_0)} |\n u|^2 + C \sqrt{\d} H_{x_0}(10),
	\]
	implying that
	\[
		\int_{A} |\n u|^2 \geq \int_{B_{10}(x_0)\sm B_{5}(x_0)} |\n u|^2 \geq (1 - \frac{1}{2^n})\int_{B_{10}(x_0)} |\n u|^2 - C \sqrt{\d} H_{x_0}(10).
	\]
	As $N_{x_0}(10) \geq 1$, $N_{x_0}(10) + V_{x_0}(10) \geq 1 - C \sqrt{\d}$, and so
	\[
		\int_{B_{10}(x_0)} |\n u|^2 \geq c H_{x_0}(10) (1 - C \sqrt{\d}).
	\]
	We arrive at
	\[
		\int_{A} |\n u|^2 \geq c H_{x_0}(10)(1 - C \sqrt{\d}) \geq c H_{x_0}(10)
	\]
which concludes the proof provided that $\d$ has been chosen small enough.
\end{proof}

\subsection{Frequency Drop Dichotomy}

\begin{theorem}\label{thm:dichotomy}
	Let $(u, \chi)$ be a variational solution on $B_{50 r}(x)$, and fix $\d > 0$, $N_+ > 0$ and $t \in [0, r]$. Then there is an $\e = \e(N_+, \d, n) \in (0,1)$ such that the following holds: set
	\[
		\maxN = \max_{y \in \overline{B_r(x)} \cap \Sigma^H_{t}} N_y(u; 20 r) \leq N_+.
	\]
	Then either $\maxN < 1 + \d$, or the set
	\[
		E = \{ y \in \overline{B_r(x)} \cap \Sigma^H_{t} : \maxN - N_y(u; r) < \e \}
	\]
	is contained in a $\d r$-neighborhood of an $n-2$-dimensional affine subspace $L \ss \R^n$.
\end{theorem}

\begin{proof}
	After a rescaling and translation, it is enough to consider $r = 1$ and $x = 0$. We assume that neither alternative holds for a sequence $(u_k, \chi_k)$ of variational solutions and numbers $\e_k \rightarrow 0$, and arrive at a contradiction. We may assume that the numbers $\maxN_k \rightarrow \bar{N} \geq 1 + \d$. Applying Lemma \ref{lem:changeofpoint} and Lemma \ref{lem:hdoubling}, for any $y \in B_1 \cap \Sigma^H_1(u_k)$ and $s \in [1, 20]$, we have $\frac{1}{C}H_0(u_k; 50) \leq H_y(u_k; s) \leq C H_0(u_k; 50)$.

	For each $k$, the set $E_k = \{ y \in \bar{B}_1 \cap \Sigma^H_{1}(u_k) : \maxN_k - N_y(u_k; 1) < \e_k \}$ fails to be contained in a $\d$-neighborhood of any $n-2$-dimensional affine subspace. We may inductively select points in $E_k$ as follows: select two points $a_0^k$, $a_1^k\in E_k$ which are a distance of at least $\d$ apart (if impossible, then $E_k \ss B_\d(a_0)$). Then find $a_2^k\in E_k$ a distance at least $\d$ from $a_0^k + \text{span}\{a_1^k - a_0^k\}$, and so on until $a_{n-1}$. Let us, for each $k$, select an orthonormal basis $e_i^k$ for $\R^n$ so that $e_1^k$ is parallel to $a_1^k - a_0^k$, $e_2^k$ is in the span of $a_1^k - a_0^k$ and $a_2^k - a_0^k$, and so on. It is straightforward to check that each $e_i^k$ can be expressed as a linear combination of the $a_i^k - a_0^k$ with coefficients of absolute value $\leq \frac{1}{\d}$; we then, for each $k$, perform an isometry so that $e_i^k = e_i$ are all the same (all quantities are invariant under isometries);
here $(e_i)_{i=1}^n$ is the standard basis of $\R^n$.

	At each of these points, we have that $N_{a_i^k}(u_k; 20) - N_{a_i^k}(u_k; 1) < \e_k$. Integrating \eqref{eq:nprime2} and also using that $|\bar{N} - N_{a_i^k}(u_k; 20)| \leq \e_k + o_k(1)$,
	\[
		\frac{1}{H_{a_0^k}(u_k; 18)} \int_{B_{20}(a_i^k) \sm B_1(a_i^k)} |\n u_k(y) \cdot y - \bar{N} u(y)|^2 = o_k(1).
	\]
	Now apply Lemma \ref{lem:twopointhom} $n-1$ times, centered at $a_0^k$ and using $a_i^k$ as the second point, to get that
	\[
		\frac{1}{H_{a_0^k}(u_k; 18)} \int_{B_{18}(a_0^k) \sm B_4(a_0^k)} |\n u_k(y) \cdot (a_i^k - a_0^k)|^2 = o_k(1).
	\]
	From the choice of basis for $\R^n$, this implies that
	\[
		\sum_{i = 1}^{n - 1} \frac{1}{H_{a_0^k}(u_k; 18)} \int_{B_{18}(a_0^k) \sm B_4(a_0^k)} |\p_i u_k(y)|^2 = o_k(1).
	\]

	Consider the sequence of renormalized solutions $v_k(y) = \frac{u_k(a_0^k + 18 y)}{\sqrt{H_{a^k_0}(u_k; 18)}}$. These satisfy the hypotheses of Lemma \ref{lem:renormalizedlimits} and Theorem \ref{thm:approx1Dlimits}, so in particular
	\[
		\lim_{k\rightarrow \infty} N_{a_0^k}(u_k; 9) = 1.
	\]
However by construction this limit also equals
$\bar{N} \geq 1 + \d$, which is a contradiction.
\end{proof}

\section{The Naber-Valtorta procedure} \label{sec:NV}

Given a variational solution on $B_{25r}(x)$ with $B_r(x) \cap \Sigma^H_t \neq \varnothing$ and $t \in [0, r]$, define $\maxN_t(x, r) = \max\{N_x(u; 20r) : x \in \Sigma^H_t \cap \overline{B_r(x)} \}$.

\begin{lemma}\label{lem:volestimateN1}
	Let $(u, \chi)$ be a variational solution on $B_{25r}(x)$, with $B_{r}(x) \cap \Sigma^H_t \neq \varnothing$ for some $t \in [0, r]$. Then there is a $\d = \d(n)$ and $C_* = C_*(n)$ such that if $\maxN_t(x, r) < 1 + \d$, we have:
	\begin{enumerate}
		\item For each $s \in [t,r]$, every countable collection of disjoint balls $B_s(x_i)$ with $x_i \in \Sigma^H_t \cap B_{r}(x)$ consists of at most $C_* \frac{r^{n-1}}{s^{n-1}}$ balls.
		\item For each $s \in [t,r]$, the volume of the $s$-neighborhood of $\Sigma^H_t \cap B_{r}(x)$, $B_s(\Sigma^H_t \cap B_{r}(x))$, is at most $C_* s r^{n-1}$.
		\item If $t = 0$, then $\cH^{n-1}(\Sigma^H \cap B_{r}(x)) \leq C_* r^{n-1}$ and $\Sigma^H \cap B_{r}$ is countably $\cH^{n-1}$-rectifiable.
	\end{enumerate}
\end{lemma}
\begin{proof}
	We rescale as usual so that $x = 0$ and $r = 1$. When proving the first two conclusions, we may as well assume that $s < s_0(n)<<1$, for otherwise they follow trivially by choosing $C(n)$ large enough. We will prove (1) in the following way: we show that if it holds for some $s_*$ and every $s \geq s_*$, then it also holds (with the same constant $C_*$, which will be chosen below) for $s = \frac{1}{20}s_*$. Then by induction, this implies (1) for every $s \geq t$.

	Take $U = \{B_s(x_i)\}$ a disjoint collection of balls with $x_i \in \Sigma^H_t \cap B_1$, and let $\mu = \sum_i s^{n-1} \d_{x_i}$ be an associated discrete measure. Then from Theorem \ref{thm:l2subspaceapprox},
	\[
		\b_\mu^2(y, \rho) \leq \frac{C}{\rho^{n-1}} \int_{B_\rho(y)} [N_z(20 \rho) - N_z(\rho)] \> d\mu(z)
	\]
	for every $t \leq \rho \leq 1$ and $y \in \Sigma^H_t \cap B_1$. On the other hand, if $\rho < s$, we instead have 
	\begin{equation}\label{eq:volestimateN1i3}
		\b_\mu(y, \rho) = 0,
	\end{equation}
	as at most one of the $x_i$ can be contained in $B_\rho(y)$ (the distance between any two of them is at least $2s$ by assumption), and so taking any $n-1$-dimensional affine hyperplane passing through that $x_i$ in the definition of $\b_\mu(y, \rho)$ makes it vanish.

	Using the elementary inequality $\b_\mu(y, a \rho) \leq C(n)\b_\mu(y, \rho)$ for $a \in [\frac{1}{20}, 1]$,
	\begin{align*}
		\int_0^1 \b_\mu^2(y, \rho) \frac{d\rho}{\rho} &\leq C \sum_{k = 0}^\infty \b_{\mu}^2(y, 20^{-k})\\
		& \leq C \sum_{k = 0}^K 20^{k (n - 1)} \int_{B_{20^{-k}}(y)}[N_z(20 \cdot 20^{-k}) - N_z(20^{-k})] \> d\mu(z)
	\end{align*}
	where $K$ is the integer with $20^{- K - 1} < s \leq 20^{ - K}$. 
	
	We integrate this with respect to $y$ with respect to the measure $\mu$, starting by estimating the integrals on the right with Fubini's theorem:
	\begin{align*}
		\int\int_{B_{20^{-k}}(y)}&[N_z(20^{1-k}) - N_z(20^{-k})] \> d\mu(z)d\mu(y) \\
		&= \int \int\chi_{B_{20^{-k}}(y)}(z) \> d\mu(y)[N_z(20^{1-k}) - N_z(20^{-k})] \> d\mu(z) \\
		&\leq \sup_{y \in B_1} \mu(B_{20^{-k}}(y)) \int [N_z(20^{1-k}) - N_z(20^{-k})] \> d\mu(z).
	\end{align*}
	Our hypothesis that (1) holds with $s \geq s_*$ can be applied to the scaled variational solution $\frac{u(y + 20^{-k} \cdot)}{20^{-k}}$ and the collection of balls $\{B_{20^{k}s}(20^{k}(x_i - y)) : x_i \in B_{20^{-k}(y)}\}$, as long as $k \geq 1$; this tells us that there are at most 
$C_* (20^{k}s)^{1 - n}$ of the points $x_i$ in $B_{20^{-k}}(y)$, for any $y \in B_1$. This leads to a simple estimate
	\begin{equation} \label{eq:volestimateN1i1}
		\mu(B_{20^{-k}}(y)) \leq C_* (20^{k}s)^{1 - n}s^{n-1} \leq C_* 20^{k (1 - n)} \textrm{ for } k \geq 1.
	\end{equation}
	When $k = 0$, a similar estimate still holds, with a larger constant, by covering $B_1 \cap \Sigma^H_t$ by $C(n)$ balls $B_{1/20}(y_i)$ centered at points in $B_1 \cap \Sigma^H_t$, and on each of those balls applying \eqref{eq:volestimateN1i1}. This gives
	\begin{equation} \label{eq:volestimateN1i2}
		\mu(B_1(y)) \leq \mu(B_1) \leq C(n) C_*.
	\end{equation}
	It follows that
	\[
		20^{k(n-1)} \int\int_{B_{20^{-k}}(y)}[N_z(20^{1-k}) - N_z(20^{-k})] \> d\mu(z)d\mu(y) \leq C C_* \int_{B_1}[N_z(20^{1-k}) - N_z(20^{-k})] \> d\mu(z).
	\]

	Summing this up,
	\begin{align*}
		\int_{B_1}\int_0^1 \b_\mu^2(y, \rho) \frac{d\rho}{\rho}d\mu(y) & \leq C C_* \sum_{k = 0}^K\int_{B_1}[N_z(20^{1-k}) - N_z(20^{-k})] \> d\mu(z) \\
		& = C C_* \int_{B_1} [N_z(20) - N_z(20^{-K})] \> d\mu(z) \\
		& \leq C C_* \int_{B_1} [N_z(20) - 1] \> d\mu(z) \\
		& \leq C C_*^2 \d,
	\end{align*}
	where the last step used our hypothesis that $N_z(20) \leq 1 + \d$ for $z \in \Sigma^H_t \cap B_1$ and \eqref{eq:volestimateN1i2}. Given any $B_\s(p)$ with $\s \in [s, 1]$ which has $\mu(B_\s(p)) > 0$, we may also perform the same estimate to the rescaled variational solution $\frac{u(p + \s \cdot)}{\s}$ and the balls $B_{s/\s}((x_i - p)/\s)$, which leads to
	\[
		\int_{B_{\s}(p)} \int_0^\s \b_\mu^2(y, \rho) \frac{d\rho}{\rho}d\mu(y) \leq C C_*^2 \d \s^{n-1}
	\]
	after scaling back. Note that the same estimate is still valid for $\s < s$, due to \eqref{eq:volestimateN1i3}. We now choose 
$C_* = D(n)$ of
the discrete Reifenberg theorem \cite{NV}[Theorem 3.4]
and we choose $\d$ small enough to obtain from \cite{NV}[Theorem 3.4] that 
	\[
		\#(U) \leq C_* (s_*/20)^{1 - n},
	\]
completing our inductive step.

	The conclusion (2), with a larger constant, is immediate from (1) by using a Vitali cover. The Hausdorff measure estimate in (3) follows from (2) applied with $t = 0$, which gives the Minkowski content bound $|B_s(B_1 \cap \Sigma^H)|\leq C_*s$.

	To see the rectifiability part of (3), let $\nu = \cH^{n-1}\mres (\Sigma^H \cap B_1)$. Applying the Hausdorff measure bound tells us that on any ball $B_\s(p)$ with $\s \leq 1$,
	\[
		\nu(B_\s(p)) \leq C_* \s^{n-1}.
	\]
	For any $y \in \Sigma^H \cap B_1$ and $\r \leq 1$, applying Theorem \ref{thm:l2subspaceapprox} leads to
	\[
		\b^2_\nu(y, \r) \leq \frac{C}{\r^{n-1}} \int_{B_\r(y)} [N_z(20 \r) - N_z(\r)] \> d \nu(z).
	\]
	Estimating similarly to $\mu$ above,
	\begin{align*}
		\int_{B_1} \int_0^1 \b^2_\nu(y, \r) \frac{d\r}{\r} \> d\nu(y) &\leq C \sum_{k = 0}^\infty \sup_{y \in B_2}\frac{\nu(B_{20^{-k}}(y))}{20^{k(n-1)}}\int_{B_1} [N_z(20^{1-k}) - N_z(20^{-k})] \> d \nu(z) \\
		& \leq C C_* \int_{B_1} [N_z(20) - N_z(0+)] \> d\nu(z)\\
		& \leq C C_*^2 \d.
	\end{align*}
	Applying to rescalings of $u$,
	\[
		\int_{B_{\s}(p)} \int_0^\s \b^2_\nu(y, \r) \frac{d \r}{\r} \> d\nu(y) \leq C C_*^2 \d \s^{n-1},
	\]
	so as long as $\d$ is chosen small enough we may apply the rectifiable Reifenberg theorem \cite{NV}[Theorem 3.3] to conclude.
\end{proof}

It is also possible to show a packing estimate similarly to \cite{NV}, but we avoid the added technicality.

We are now in a position to prove the main theorem of this and the previous section.
\begin{theorem}\label{thm:main}
	Let $(u, \chi)$ be a variational solution on $B_1$, with $0 \in \Sigma^H_t$ 
for some $t\in [0,1]$
and $N_0(u; 1) \leq N_+$. Then there is a constant $C = C(N_+, n)$ such that the following holds:
	\begin{enumerate}
		\item For every $s \geq t$, $|B_s(\Sigma^H_t \cap B_{1/100}(x))| \leq C s$.
		\item If $t = 0$, then $\cH^{n-1}(\Sigma^H \cap B_{1/100}(x)) \leq C$ and $\Sigma^H \cap B_{1/100}$ is countably $\cH^{n-1}$-rectifiable.
		\item $N_x(u; 0+) = 1$ at $\cH^{n-1}$-a.e. $x \in \Sigma^H$.
	\end{enumerate}
\end{theorem}

\begin{proof}
	By Lemma \ref{lem:changeofpoint}, $\maxN_t(0, 1/50) \leq C N_+$.

	We will proceed by building, inductively, a sequence of coverings of $\Sigma^H_t$ by balls. We start with the ball $B_r(x) = B_{1/50}$ (with $x \in \Sigma^H_t$ and $r \geq t$), and check the value of $\maxN_t(x, r)$. If this quantity is less than $1 + \d_1$ (with $\d_1$ fixed and smaller than the constant $\delta$ in Lemma \ref{lem:volestimateN1}), we call this a \emph{terminal} ball, add it to a collection $W^T$, and do not subdivide it any further. If not, we apply Theorem \ref{thm:dichotomy} with $\d_2 \leq \min\{\d_1, \frac{1}{20}\}$ to be selected below, to see that
	\[
		E = \{y \in B_r(x) \cap \Sigma^H_t : \maxN_t(x, r) - N_y(r) < \e \} \ss B_{\d_2 r}(V),
	\]
	where $V$ is an $n-2$-dimensional affine subspace. 
	
	We cover $B_r(x) \cap \Sigma^H_t$ by a large number of balls $B_{\d_2 r}(x_i)$, with $x_i \in B_r(x) \cap \Sigma^H_t$ and the balls having finite overlap. On these, there are two possibilities for the behavior of $\maxN_t(x_i, \d_2 r)$. If $B_{\d_2 r}(x_i)$ is disjoint from $E$, then for every $y \in B_{\d_2 r}(x_i) \cap \Sigma^H_t$,
	\[
		N_{y}(20 \> \d_2 r) \leq N_y(r) \leq \maxN_t(x, r) - \e,
	\]
	so $\maxN_t(x_i, \d_2 r) \leq \maxN_t(x, r) - \e$; we say these balls are ones with \emph{large frequency drop}. The other possibility is that $B_{\d_2 r}(x_i)$ intersects $E$, in which case we only know $\maxN_t(x_i, \d_2 r) \leq \maxN_t(x, r)$; these balls have \emph{small frequency drop}.

	Starting with $B_{1/50}$ and assuming that $B_{1/50}$ is not terminal, we apply this procedure to get a collection $U^1$ of balls, and then apply it to all the balls in $U^1$ to get a collection $U^2$, and so on. Each of these collections $U^k$ has only balls of radius $r_k = \d_2^k \frac{1}{50}$ and finite overlap, and we continue on until $r_{k+1} < s$, 
	so we end up with a finite number of balls for each $s$. Setting $K := \inf\{ k\geq 1: r_k\geq s\}$, we obtain $r_K\in [s, s/\d_2]$ as the radius of the final generation of non-terminal balls. The collection of terminal balls may have balls of varying radii, each greater than or equal to $s$.

	Let us estimate the $q$-dimensional size of the balls in $U^k$. In each generation, a ball is subdivided into at most $C_L \d_2^{-n}$ balls, of which at most $C_S \d_2^{2 - n}$ can have small frequency drop, due to the fact that $E\subset B_{\delta_2 r}(V)$; importantly, $C_S, C_L$ depend only on $n$ (and not on $\d_2$, or $\e(\d_2)$). If we track each ball's ancestry, we see that it could not have had large frequency drop more than $D := D(n, N_+, \d_2) := \lceil C N_+ /\e \rceil$ times: otherwise that would contradict that $N_{x_i}(t) \geq 1$ (indeed, one of its ancestors would have been terminal). So the total number of balls in generation $U^k$ is bounded by
	\[
		\#(U^k) \leq [C_S \d_2^{-n}]^D [C_L \d_2^{2-n}]^{k} \leq C(C_S, D, \delta_2) C_L^k \d_2^{(2 - n) k} \leq C(n, N_+, \d_2) r_k^{2 - n - \t},
	\]
	where we choose any $\t \in (0, 1)$ and then $\d_2$ small enough in terms of $\t$ and $n$ only so that $C_L < \d_2^{-\t}$. Therefore,
	\begin{equation}\label{tau1}
		\sum_{B_{r_k}(x_i) \in U^k} r_k^{q} = r_k^{q} \> \#(U^k)  \leq C r_k^{q + 2 - n - \t}.
	\end{equation}

Let us define the collection $W^s := \bigcup_{k\geq 1}U^k$ of all balls in all generations
having radii $\geq s$ as well as the subcollection $W^T$ of all terminal balls in $W^s$.
Provided that $q > n - 2 + \t$, 
we obtain that
\begin{equation}\label{tau2}
		\sum_{B_{r_i}(x) \in W^s} r_i^q \leq C \sum_{k = 0}^\infty \d_2^{(q + 2 - n - \t) k} \leq C.
	\end{equation}

	We are now ready to prove the theorem. For (1), note that
	\[
		B_s(\Sigma^H_t \cap B_{1/100}(x)) \ss \bigcup_{B_{r_K}(x_i) \in U^K} B_{2 r_K}(x_i) \cup \bigcup_{B_{r_i}(x) \in W^T}B_s(\Sigma^H_t \cap B_{r_i}(x)).
	\]
	Estimating the volumes, the first union has, using \eqref{tau1} with $q=n$,
	\[
		\left|\bigcup_{B_{r_K}(x_i) \in U^K} B_{2 r_K}(x_i)\right| \leq C r_K^{2 - \t} \leq C s^{2 - \t},
	\]
	while the second, using Lemma \ref{lem:volestimateN1} as well as \eqref{tau1} with $q=n-1$, satisfies
	\[
		\left|\bigcup_{B_{r_i}(x) \in W^T}B_s(\Sigma^H_t \cap B_{r_i}(x))\right| \leq \sum_{B_{r_i}(x) \in W^T} |B_s(\Sigma^H_t \cap B_{r_i}(x))| \leq C s \sum_{B_{r_i}(x) \in W^T} r^{n-1}_i \leq C s.
	\]
	This establishes (1).

	For (2), the Hausdorff measure estimate follows immediately from (1) using that $\Sigma^H \ss \Sigma^H_t$. If we denote by $A$ the portion of $\Sigma^H$ which is not contained in the union of the terminal balls, we have shown that in fact $\cH^{n-2 + \t}(A) \leq C$, so $\cH^{n-1}(A) = 0$. The rest of $\Sigma^H$ is contained in a countable union of countably $\cH^{n-1}$-rectifiable sets by Lemma \ref{lem:volestimateN1} applied to each terminal ball, and so is countably $\cH^{n-1}$-rectifiable.

	Finally, for (3) observe that on $\Sigma^H \cap B_{r_i}(x_i)$, where $B_{r_i}(x_i)$ is a terminal ball, we have by construction that $\maxN_0(x_i, r_i) \leq 1 + \d_1$, and so $N_x(u; 0+) \leq 1 + \d_1$ on $B_{r_i}(x_i)$. This is true on all terminal balls, which cover $\Sigma^H \sm A$, so $N_x(u; 0+) \leq 1 + \d_1$ for $\cH^{n-1}$-a.e. point in $\Sigma^H$. This is true for every $\delta_1$ small, so as countable unions of $\cH^{n-1}$-null sets are $\cH^{n-1}$-null sets, (3) follows.
\end{proof}

\section{A complete identification of the \texorpdfstring{measure $\Delta u$}{Laplacian measure} of variational solutions}\label{sec:corollaries}
Theorem \ref{thm:main} allows us to completely identify $\Delta u$ of
 variational solutions; more precisely (see Theorem \ref{thm:distributional}),
\[\Delta u = \cH^{n-1}\mres \p^* \{ u > 0\} + 2 \a(n) \sqrt{H_x(u; 0+)} \cH^{n-1}\mres \Sigma^H.\]
In later sections we will in turn apply this result to any limit 
of classical solutions of the Bernoulli problem (see Section \ref{sec:compactness})
and any limit of the related singular perturbation problem (see Section \ref{sec:perturbation}).
In particular, the {\em full topological free boundary of any of those limits
is countably $\cH^{n-1}$-rectifiable and has locally finite $\cH^{n-1}$-measure.}
This result also gives information about the harmonic measure
of the set $\{u > 0\}$ (see Remark \ref{rem:harmonicmeasure}).

In preparation of Theorem \ref{thm:distributional} we need some analysis of the
regular free boundary (Section \ref{sec:anreg}), some analysis of singular points
(Section \ref{sec:ansing}) as well as an identification of the limit $\chi$
(Section \ref{sec:identchi}). 

\subsection{Analysis of regular points}\label{sec:anreg}

Let $(u, \chi)$ be a variational solution on $\W$. By the upper semicontinuity of $x\mapsto M_x(u; 0+)$ (see Proposition \ref{prop:Mlimits}), the set $\Sigma^H$ is relatively closed, so $\W \sm \Sigma^H$ is an open set. In this section we recall some results on $\W \sm \Sigma^H$, and in particular on the portion of the free boundary $\p \{u > 0\}$ contained in it (this is exactly the portion of the free boundary with density $M_x(0+) < |B_1|$). These mostly follow existing literature, such as \cite{CS,C87,C89,D}.

To begin with, observe that in the interior of $\{u = 0\}$, $\chi$ must be locally constant. Indeed, \eqref{eq:variationalformula} gives
\[
	\int \chi \dvg \xi = 0
\]
for all $\xi \in C_c^\infty(\{u = 0\}^\circ, \R^n)$, which implies that $\chi$ is constant on each connected component. If $\chi = 0$ in a component, we have the ``natural'' behavior $M_x(u; 0+) \equiv 0$ there (by Proposition \ref{prop:Mlimits}), while if $\chi = 1$, we have the ``pathological'' case of $M_x(u; 0+) \equiv |B_1|$ instead. The latter is certainly possible: $(u, \chi) = (0, 1)$ is a variational solution. However, $x\mapsto M_x(u; 0+)$ is upper semicontinuous (see Proposition \ref{prop:Mlimits}), so at any $x \in \partial \left(\{u = 0\}^\circ\right)$ at the boundary of any such pathological connected component, we would also have $M_x(u; 0+) = |B_1|$, and so $x\in \Sigma^H$. In Section \ref{sec:ansing}, we will show this is actually impossible, and so $\chi = 0$ on $\{u = 0\}^\circ$ unless $u \equiv 0$. In this section we simply note that on a connected open set with $\Sigma^H = \varnothing$, unless $u \equiv 0$ we have that $M_x(u; 0+) < |B_1|$.

It will be helpful to have the following notion of \emph{viscosity solution} of our problem.
\begin{definition}[Viscosity Solution]
Let $u$ be a continuous function $u : \W \rightarrow [0, \infty)$ which is harmonic in $\{u > 0\}$. Then $u$ is a viscosity solution if any smooth function $\phi$ satisfying $\phi \leq u$ ($\phi_+ \geq u$) in $B_r(x) \ss \W$, $\phi(x) = 0$ and $x \in \p \{u > 0\}$ also satisfies $|\n u| \leq 1$ ($|\n u| \geq 1$).
\end{definition}

It is easy to see that even when $n = 1$, the function $x_+ + \frac{1}{2}x_-$ is a viscosity solution but not a variational solution (for any $\chi$), whereas $2|x_n|$ is a variational solution but not a viscosity solution. However, outside the set $\Sigma^H$ variational solutions are viscosity solutions (this will be shown below in Lemma \ref{lem:regvarisvisc}), and on this set viscosity solutions admit a partial regularity theorem (see Corollary \ref{cor:regular}).

\begin{lemma}\label{lem:density0char}
	Let $(u, \chi)$ be a variational solution in $\W$, and $B_{2r}(x) \cc \W$. Then if $u(x) = 0$,
	\[
		M_x(u; 0+) \in \{0\} \cup [|B_{1}|/2, |B_1|],
	\]
	with $M_x(u; 0+) = 0$ only if $x$ is in the interior of $\{u = 0\}$.
Moreover, if $x\in \p\{ u>0\}$ and the blow-up sequence $u(x+r\cdot)/r\to 0$ as $r\to 0$,
then $M_x(u; 0+)=|B_1|$.
\end{lemma}

\begin{proof}
	Assume that $x = 0$. Let us use the same notation as the proof of Proposition \ref{prop:Mlimits}, with blow-ups of $u$ along a subsequence converging to a homogeneous function
$u_0$ of degree $1$. We first show that $M_0(u; 0+) = 0$ may only happen if $u_0 \equiv 0$ (along each convergent subsequence), and $M(u; 0+) \geq |B_1|/2$ otherwise. 
	
	Indeed, if $u_0$ is nonzero, then its trace on $\p B_1$ is a solution to 
	\[
		- \Delta_{S^{n-1}} u_0 = (n - 1) u_0 \qquad \text{on } \{ u_0 > 0\}.
	\]
	In other words, it is a first eigenfunction of the set $\{ u_0 > 0\} \cap \p B_1$ with eigenvalue $n-1$. An inequality of Sperner \cite{S} guarantees that any such set must have volume greater than or equal to that of a spherical cap with the same eigenvalue, which in this case is a half-sphere with volume $|\p B_1|/2$. As $u_0$ is $1$-homogeneous, we see that $|\{u_0 > 0\} \cap B_1| \geq |B_1|/2$, and so
	\[
		|B_1|/2 \leq |\{u_0 > 0\} \cap B_1| \leq  \liminf_{k\to\infty} \frac{|\{u_{r_k} > 0\}|}{r_k^n} \leq \lim_{k\to\infty} \frac{|\{\chi_k = 1\}|}{r_k^n} = M_0(u; 0+).
	\]

	The function $x \mapsto M_x(u; 0+)$ is upper semicontinuous on $\W$ (see Proposition \ref{prop:Mlimits}).
Therefore $\{x: M_x(u; 0+) \leq 0\} = \{x : M_x(u; 0+) < |B_1|/4\}$ is relatively open. Recall that $M_x(u; 0+) = - \infty$ if and only if $u(x) > 0$, and so this set is the union of $\{u > 0\}$ and $\{x : M_x(u; 0+) = 0\}$. Take a point $x$ with $M_x(u; 0+) = 0$, and assume that $M_y(u; 0+)\leq 0$ for $y \in B_\d(x)$ for a small $\d$. Then at every point $y \in B_\d(x)$ we must either have $M_y(u; 0+) = 0$ and hence
	\[
		\lim_{s \to 0} \frac{|\{\chi = 1\} \cap B_s(y)|}{s^n} = 0,
	\]
	or $u(y) > 0$ and so
	\[
		\lim_{s \to 0} \frac{|\{\chi = 1\} \cap B_s(y)|}{s^n} = |B_1|.
	\]
	This, however, implies that the essential boundary of $\{\chi = 1\}$ is empty in $B_\d(x)$, and so by a theorem of Federer we must have either $\chi \equiv 1$ on $B_\d(x)$ or $\chi \equiv 0$ on $B_\d(x)$, Lebesgue a.e. (indeed, \cite{F}[4.5.11] implies $\{\chi = 1\}$ is a set of finite perimeter; \cite{F}[4.5.6(1)] implies its Gauss-Green boundary measure is trivial, being supported on the essential boundary; the relative isoperimetric inequality \cite{F}[4.5.3] then gives the claim). The former is impossible, since at least one point, $x$, has density zero: this means $\chi = 0$ on a set of positive Lebesgue measure. The latter means that $M_y(u; 0+) = 0$ and $u(y) = 0$ for Lebesgue a.e. $y \in B_\d(x)$, and so by continuity $u = 0$ on all of $B_\d(x)$.

The last statement follows from the fact that
$u(x+r\cdot)/r\to 0$ as $r\to 0$ implies the limit of $\chi(x+r\cdot)$ as $r\to 0$
must be a constant function. As by the already proved part of the lemma,
$M_x(u; 0+) \ne 0$, it follows that $M_x(u; 0+)=1$.
\end{proof}

\begin{lemma} \label{lem:regperimeterest}
	Let $(u, \chi)$ be a variational solution on a connected open set $\W$, and assume $\Sigma^H = \varnothing$. Then either $u \equiv 0$ or $\chi = \chi_{\{u > 0\}}$ a.e., so the set $\{u > 0\}$ has locally finite perimeter. Moreover,
	\[
		\cH^{n-1}(\p \{u > 0\} \cap B_r(x)) \leq C(C_V) r^{n-1}
	\]
	for any $B_{2r}(x) \cc \W$.
\end{lemma}

\begin{proof}
	As discussed at the start of this section, we have that $\chi = \chi_{\{u > 0\}}$ on $\W \sm \p \{u > 0\}$, and on $\p \{u > 0\}$ we have $M_x(0+) \in [|B_1|/2, |B_1|)$. In particular, on $\p \{u > 0\}$ the Lebesgue density of $\{\chi = 1\}$ lies in $[1/2, 1)$, and so $\p \{u > 0\}$ is contained in the essential boundary $\partial^e\{\chi = 1\}$. From Lemma \ref{lem:BV}, this implies that
	\[ 
		\cH^{n-1}(\p \{u > 0\} \cap B_r(x)) \leq \cH^{n-1}(\partial^e\{\chi = 1\} \cap B_r(x)) = \int_{B_r(x)}|\n \chi| \leq C(C_V) r^{n-1}
	\]
	for any $B_{2r}(x) \cc \W$ (using Federer's theorem \cite{F}[4.5.6(1)]). This implies that $|\p \{u > 0\}| = 0$, $\chi = \chi_{\{u > 0\}}$ a.e., and that $\{u > 0\}$ has finite perimeter.
\end{proof}

\begin{lemma} \label{lem:regvarisvisc}
	Let $(u, \chi)$ be a variational solution on $\W$, and assume $\Sigma^H = \varnothing$. Then $u$ is a viscosity solution on $\W$.
\end{lemma}

\begin{proof}
	Let $\phi \leq u$ ($\phi_+ \geq u$) on $B_r(x)$, $\phi(x) = u(x) = 0$, and $x \in \p \{u > 0\}$, and assume without loss of generality that if the first inequality holds then $|\n \phi(x)|\neq 0$. The fact that $x \in \p \{u > 0\}$ implies that $M_x(u; 0+) \geq |B_1|/2$ by Lemma \ref{lem:density0char}, while as $\Sigma^H = \varnothing$ we have $M_x(u; 0+) < |B_1|$ (recall that $\Sigma^H$ is exactly the portion of $\p \{u > 0\}$ with $M_x(u; 0+) = |B_1|$).
	
	Then as in the proof of Proposition \ref{prop:Mlimits}, along a sequence $s_k$ we have that
	\[
		u_k(y) = \frac{u(x + s_k y)}{s_k} \rightarrow u_0(y)
	\]
	locally uniformly and strongly in $W^{1,2}$, where $u_0$ is a $1$-homogeneous function which is harmonic on $\{u_0 > 0\}$. By Lemma \ref{lem:BV}, if $\chi_k(y) = \chi(x + s_k y)$, then $\|\chi_k\|_{BV(B_1)} \leq C$, and so up to a subsequence $\chi_k \rightarrow \chi_0 \in BV$ in $L^1_{\mathrm{loc}}$ and a.e.. We may use the $W^{1, 2}$ strong convergence to also pass the variational identity \eqref{eq:variationalformula} to the limit in the following sense: for every $\xi \in C_c^{0, 1}(\R^n; \R^n)$,
	\[
		\int \left(|\n u_0|^2 \dvg \xi - 2 \n u_0 \cdot D\xi \n u_0\right) = - \int \chi_0 \dvg \xi.
	\]
	We also know that $\chi_{\{u_0 > 0\}} \leq \liminf_{k\to\infty} \chi_{\{u_k > 0\}} \leq \chi_0$ a.e., using only that $\chi_{\{u > 0\}} \leq \chi$ and pointwise convergence.
	
	On the other hand, clearly
	\[
		\lim_{s_k \to 0+} \frac{\phi(x + s_k y)}{s_k} = \n \phi(x) \cdot y,
	\]
	and so, by homogeneity of $u_0$, $\n \phi(x) \cdot y \leq u_0(y)$ ($(\n \phi(x) \cdot y)_+ \geq u_0(y)$). Choose coordinates so that $\n \phi(x) \cdot y = \b y_n$, $\b \geq 0$. These inequalities imply that $u$ has the form $u(y) = \a_+ (y_n)_+ + \a_- (y_n)_-$ for two numbers $\a_+, \a_-$, with $\b \leq \a_+$ ($a_+ \leq \b$). Indeed, the trace $u|_{\p B_1}$ is a first eigenfunction of the Laplace-Beltrami operator on each connected component of the domain $\p B_1 \cap \{u > 0\}$ with eigenvalue $n-1$, from a standard separation of variables argument. From Sperner's inequality \cite{S}, each component must have volume at least $|B_1|/2$, with equality if and only if it is a half-sphere. If a component $U$ of $\{u > 0\}$ is contained in half-sphere (so when $\b (y_n)_+ \geq u_0$, for example), this implies that $\{u > 0\}$ is that half-sphere and $u$ has the form $\a (y_n)_+$ on $U$ (the explicit first eigenfunction of the Laplace-Beltrami operator on it). If, on the other hand, a component $U$ of $\{u > 0\}$ contains a half-sphere (so when $\b y_n \leq u_0$ with $\b > 0$), one possibility is that there is also a second connected component $U'$ contained in the complementary half-sphere; then it follows that both are half-spheres and $u$ has the required form from our previous analysis. On the other hand, if $U = \{u > 0\}$ is the only connected component, it is easy to see from the domain monotonicity of the first eigenvalue and the fact that $U$ is open that $U$ is a half-sphere as well.
	
	By Proposition \ref{prop:Mlimits}, we must have that both $u_0$ and $\chi_0$ vanish on a set of positive Lebesgue measure, and that $\chi_0 = 1$ on a set of positive measure. In particular, this means at least one of $\a_\pm$ must be $0$, and we see that in both cases it must be $\a_-$ which is $0$.

	We can now directly compute
	\begin{align*}
		\int \left(|\n u|^2 \dvg \xi - 2 \n u \cdot D\xi \n u\right) &= \a_+^2 \int_{\{y_n > 0\}} \left(\dvg \xi - 2 \frac{\partial \xi_n}{\partial x_n}\right) \\&= - \a_+^2 \int_{\{y_n > 0\}} \frac{\partial \xi_n}{\partial x_n} = \a_+^2 \int_{\{y_n = 0\}} \xi_n\> d\mathcal{H}^{n-1}.
	\end{align*}
	This completely determines $\chi_0$, as, first, it implies that $\n \chi_0$ is supported on $\{y_n = 0\}$, and hence $\chi_0$ is constant on the two half-spaces. But then
	\[
		\a_+^2 \int_{\{y_n = 0\}} \xi_n\> d\mathcal{H}^{n-1} = - \int \chi_0 \dvg \xi =  \int_{\{y_n = 0\}} \left[\chi_0|_{\{y_n > 0\}} - \chi_0|_{\{y_n < 0\}}\right] \xi_n\> d\mathcal{H}^{n-1},
	\]
	from which we see that $\a_+^2 = \chi_0|_{\{y_n > 0\}} - \chi_0|_{\{y_n < 0\}}$. As $\chi_0 \in \{0, 1\}$ a.e., the only possibilities are
	\begin{align*}
		&1.\quad &\a_+ = 1, \quad & \chi_0|_{\{y_n > 0\}} = 1, \quad \chi_0|_{\{y_n < 0\}} = 0, \\
		&2.\quad &\a_+ = 0, \quad & \chi_0|_{\{y_n > 0\}} = 1, \quad \chi_0|_{\{y_n > 0\}} = 1, \\
		&3.\quad &\a_+ = 0, \quad & \chi_0|_{\{y_n > 0\}} = 0, \quad \chi_0|_{\{y_n > 0\}} = 0.
	\end{align*}	
	The second case is, in fact, impossible, as we know that $\chi_0 = 0$ on a set of positive measure, and likewise the third case is impossible as $\chi_0 = 1$ on a set of positive measure. We conclude that $a_+ = 1$, and so $\b \leq 1$ ($\b \geq 1$); this implies the conclusion.
\end{proof}

\begin{corollary}\label{cor:regular}
	Let $(u, \chi)$ be a variational solution on a connected open set $\W$, and assume $\Sigma^H = \varnothing$. Then the reduced boundary $\p^* \{u > 0\}$ is relatively open in $\p \{u > 0\}$ and is locally given by graphs of analytic functions along which $|\n u| = 1$ from the $\{u > 0\}$ side. The rest of the boundary $\p \{u > 0\} \sm \p^* \{u > 0\}$ has Hausdorff dimension at most $n - 3$. The Laplacian of $u$ satisfies 
	\[
		\Delta u = \cH^{n-1}\mres \p^*\{u > 0\}
	\]
	in the sense of distributions.
\end{corollary}
\begin{proof}
First note that by \cite[Theorem 1.1]{D},
$\bar \epsilon$-flatness the free boundary of a viscosity
solution $v$ in $B_1$ in the sense
\[(x_n-\bar \epsilon)^+ \leq v(x) \leq (x_n+\bar \epsilon)^+\]
implies the free boundary to be $C^{1,\alpha}$ in $B_{1/2}$.
Suppose now that $\mathcal{H}^{s}((\p \{u > 0\} \sm \p^* \{u > 0\})\cap D)>0$
for some $s>n-3$.
Next we observe that by the assumption $\Sigma^H = \varnothing$
as well as Lemma \ref{lem:density0char},
a limit of the blow-up family $u(x+r\cdot)/r$ as a sequence $\to 0$
at free boundary points $x$ cannot be $0$.
Then by the dimension reduction procedure in the proof of \cite[Theorem 4.5]{W99}---replacing minimizers by viscosity and variational solutions,
the monotonicity formula by Proposition \ref{prop:monotone},
the flatness class $F$ by the flatness in the above sense,
and using this class of solutions is closed---we obtain a $1$-homogeneous viscosity solution $\bar u$
in dimension $k<3$ such that
$\mathcal{H}^{s-(n-3)}((\p \{\bar u > 0\} \sm \p^* \{\bar u > 0\})\cap D)>0$.
After rotation $\bar u(x) = \a (x_n)_+ + \b (x_n)_-$.
Now, by the assumption $\Sigma^H = \varnothing$
and 
by the upper semicontinuity
Proposition \ref{prop:Mlimits},
\[\limsup_{r\to 0} |B_r(x)\cap \{\chi = 1\}|/|B_r| \leq \theta <1
\text{ for }  x\in \p \{u > 0\}\cap D,\]
where $D$ is the connected component of $\W$.
By lower semicontinuity of the characteristic function $\chi_{\{t>0\}}$ it follows
that either $\a=0$ or $\b=0$. But then 
 \cite[Theorem 1.1]{D} implies that $\p \{\bar u > 0\} \sm \p^* \{\bar u > 0\}
=\varnothing$, a contradiction.
\end{proof}
\subsection{Analysis of singular points}\label{sec:ansing}

The analysis of regular points did not require Theorem \ref{thm:main}. However, without that theorem we are faced with serious challenges in the presence of singular points: the set $\Sigma^H$, and hence the support of $\Delta u$, might be large and have rather arbitrary structure. Theorem \ref{thm:main} rules out the most problematic anomalies and allows a much finer characterization of $\Delta u$ and $\chi$.

Let
\[
	\a(n) = \frac{1}{\sqrt{\int_{\p B_1} |x_n|^2 \> d\cH^{n-1}}}
\]
be an explicit dimensional constant, which has the useful normalization property $$H_0(\a(n)|x_n|; r) \equiv 1.$$ It is the same constant as in Theorem \ref{thm:approx1Dlimits}.

\begin{lemma} \label{lem:singular1density}
	Let $(u, \chi)$ be a variational solution on $\W$. Then for $\cH^{n-1}$-a.e. $x \in \Sigma^H$,
	\begin{enumerate}
		\item the renormalized functions
		\[
			v_{x, r}(y) = \frac{u(x + r y)}{r \sqrt{H_x(u; r)}} \rightarrow \a(n) |y_n|
		\]
		strongly in $W^{1, 2}_{\mathrm{loc}}(\R^n)$.
		\item the Lebesgue density of $\{u > 0\}$ at $x$ exists and is $1$:
		\[
			\lim_{r \to 0}\frac{|B_r(x) \cap \{u > 0\}|}{|B_r|} = 1.
		\]
	\end{enumerate}
\end{lemma}

Note that the density of $\{\chi = 1\}$ must be $1$ at every point in $\Sigma^H$ from the definition of $\Sigma^H$ and Proposition \ref{prop:Mlimits}. The second part of this lemma is making a stronger and nontrivial assertion, and will be used to rule out $\Sigma^H$ being located along the boundary of a ``pathological'' component of $\{u = 0\}$ where $\chi = 1$; this will be made rigorous below.

\begin{proof}
	From Theorem \ref{thm:main}, at $\cH^{n-1}$-a.e. $x \in \Sigma^H$, we have that $N_x(u; 0+) = 1$ and $\Sigma^H$ admits an approximate tangent, i.e. the measures
	\[
		\nu_{x, r}(E) = \frac{\cH^{n-1}(E\cap \Sigma^H \cap B_r(x))}{r^{n-1}}
	\]
	converge in the weak-$*$ sense to $\cH^{n-1}\mres L$ for some hyperplane $L \ss \R^n$ (passing through the origin). We show that at any such $x$, both (1) and (2) hold. Without loss of generality translate so $x = 0$. By the continuity of $y \mapsto N_y(u; r)$ for each small and positive $r$, we may assume that for each positive $\eta$, that there is a $\d$ such that $N_y(u; r) \leq 1 + \eta$ for $y \in B_\d \cap \Sigma^H$ and $r < \d$.

	Given any sequence $r_k \to 0$, consider the renormalized sequence
	\[
		v_k(y) = v_{0, r_k}(y) = \frac{u(r_k y)}{r_k \sqrt{H_0(u; r_k)}}.
	\]
	From Lemma \ref{lem:renormalizedlimits}, $v_k \rightharpoonup v \in W^{1, 2}(B_R)$ weakly, strongly in $L^2(B_R)$, and a.e. as $k\to\infty$ for each $R>0$ (passing if necessary to a subsequence). It follows that $\chi_{\{v > 0\}} \leq \liminf_{k \rightarrow \infty}\chi_{\{v_k > 0\}}$ a.e., and so by Fatou's lemma
	\[
		\frac{|\{v > 0\} \cap B_1|}{|B_1|} \leq \liminf_{k \rightarrow \infty} \frac{|\{v_k > 0\} \cap B_1|}{|B_1|} = \liminf_{k \rightarrow \infty} \frac{|\{u > 0\} \cap B_{r_k}|}{|B_{r_k}|}.
	\]
	As $|y_n| > 0$ Lebesgue a.e., (2) will follow directly from (1).

	Due to the existence of an approximate tangent for $\nu$, for any $B_\t(y)$ with $y \in L$ we have that
	\[
		\liminf_{k \rightarrow \infty} \nu_{0, r_k}(B_\t(y)) \geq \cH^{n-1}(B_\t(y) \cap L) \geq c(n) \t^{n-1}.
	\]
	In particular, for $k$ large $B_{\t r_k}(r_k y) \cap \Sigma^H$ is nonempty. Select coordinates so that $L = \{y_n = 0\}$, and use $y = e_i/10$, $i \leq n - 1$ to obtain points $y_{i, k} \in \Sigma^H \cap B_{1/40 r_k}(r_k e_i)$. Using Lemma \ref{lem:hdoubling} and \ref{lem:changeofpoint}, we know that $\frac{1}{C}H_{0}(u; r_k) \leq  H_{y_{i, k}}(u; R r_k) \leq C H_{0}(u; r_k)$ for every $R \in [1/40, 40]$ and for some $C = C(n)$ independent of $r_k$ and $y_{i, k}$. Then from \eqref{eq:nprime2},
	\begin{align*}
		\int_{r_k/4}^{4 r_k} \int_{\p B_s(y_{i, k})}& \frac{1}{H_{0}(u; r_k) r_k^{n + 2}} |\n u(y) \cdot y - N_{y_{i, k}}(u; s r_k) u(y)|^2 \> d\cH^{n-1} ds \\
		&\leq C [N_{y_{i, k}}(u; 4 r_k) - N_{y_{i, k}}(u;  r_k/4) ] \\
		&\leq C \eta,
	\end{align*}
	using $N_{y_{i, k}} \in [1, 1 + \eta]$ for $k$ large. Using $N_{y_{i, k}} \in [1, 1 + \eta]$
once more,
	\[
		\int_{B_{4}(y_{i, k}/r_k) \sm B_{1/4}(y_{i, k}/r_k)} |\n v_k(y) \cdot y - v_k(y)|^2 \> dy \leq C \eta.
	\]
	The same estimate is true at $0$:
	\[
		\int_{B_{4} \sm B_{1/4}} |\n v_k(y) \cdot y - v_k(y)|^2 \> dy \leq C \eta.
	\]
	Applying Lemma \ref{lem:twopointhom} for each $i$, we obtain that
	\[
		\int_{B_{2} \sm B_{1/2}} |\n v_k(y) \cdot \frac{y_{i, k}}{r_k}|^2 \leq C \eta,
	\]
	which implies that
	\[
		\int_{B_{2} \sm B_{\frac{1}{2}}} |\partial_i v_k|^2 \leq C \eta
	\]
	due to how $y_{i, k}$ were chosen. As $\eta$ was arbitrary,
	\[
		\lim_{k \rightarrow \infty} \int_{B_{2} \sm B_{\frac{1}{2}}} |\partial_i v_k|^2 = 0.
	\]

	We can therefore apply Theorem \ref{thm:approx1Dlimits} to conclude that $v(y) = \a(n)|y_n|$ on $B_1\sm \overline{B_{1/2}}$. A rescaled version of the same argument shows that in fact $v(y) = \a(n) |y_n|$ on any annulus $B_R \sm \overline{B_{R/2}}$, and so on all of $\R^n$. By the same argument as in the proof of Theorem \ref{thm:approx1Dlimits}, this implies that $v_k \rightarrow v$ \emph{strongly} in $W^{1, 2}(B_R)$: for any $\eta \in C_c^\infty(\R^n)$,
	\[
		\int |\n v|^2 \eta = - 2 \int v \n \eta \cdot \n v = \lim_{k \rightarrow \infty} - 2 \int v_k \n \eta \cdot \n v_k = \lim_{k \rightarrow \infty} \int \eta |\n v_k|^2,
	\]
	where the first and last terms used that $v, v_k$ are harmonic when positive.

	As the limit $\a(n) |y_n|$ is independent of subsequence, we conclude that (1) holds.
\end{proof}
\subsection{Identification of the limit \texorpdfstring{$\chi$}{chi}}\label{sec:identchi}
\begin{theorem}\label{thm:chi}
	Let $(u, \chi)$ be a variational solution on $\W$, where $\W$ is open and connected. Then either $u \equiv 0$ or $\chi = \chi_{\{u > 0\}}$ almost everywhere.
\end{theorem}

\begin{proof}
	From Lemma \ref{lem:regperimeterest} and \ref{thm:main}, we have that $|\partial \{u > 0\}| = 0$, while by definition on $\{u > 0\}$ we have $\chi(x) = 1 = \chi_{\{u > 0\}}(x)$ a.e.; therefore it suffices to show that $\chi = 0$ a.e. on the interior $\{u = 0\}^\circ$. 
	
	Take a nonempty connected component $U$ of $\{u = 0\}^\circ$ on which $\chi = 1$ (recalling that $\chi$ is locally constant on $\{u = 0\}^\circ$). Clearly $\p U \cap \W \ss \p \{u > 0\}$, and for $x \in \p U$ we must have $M_x(0+) \geq \limsup_{y \rightarrow x} M_y(0+) = 1$, using Proposition \ref{prop:Mlimits} and the upper semicontinuity of $x \mapsto M_x(0+)$ (see Proposition \ref{prop:Mlimits}). In particular, $\p U \cap \W \ss \Sigma^H$. Applying Theorem \ref{thm:main}, $\p U \cap \W$ has locally finite Hausdorff measure, and so $U$ has locally finite perimeter.

	Assume that $U \neq \W$: then there is an smooth connected open set $\W' \cc \W$ such that $|U \cap \W'| > 0$, $|\W' \sm U| > 0$. By the relative isoperimetric inequality, then, $\cH^{n-1}(\p^* U \cap \W') > 0$, where $\p^* U$ is the reduced boundary. At each point $x \in \p^* U$, $U$ has Lebesgue density $\frac{1}{2}$, so $\{u > 0\} \ss \W \sm U$ can have Lebesgue density at most $\frac{1}{2}$. On the other hand, by Lemma \ref{lem:singular1density} we have that the Lebesgue density of $\{u > 0\}$ at $\cH^{n-1}$-a.e. point in $\Sigma^H$, which contains $\p^*U \ss \p U$, is $1$. This is a contradiction.
\end{proof}

For $x \in \Sigma^H$, we have from Lemma \ref{lem:hdoubling} that $H_x(u; r)$ is nondecreasing, so the limit $H_x(u; 0+) \in [0, \infty)$ exists. It is clear that in fact $H_x(u; 0+) \leq C_V^2 |\p B_1|$, from the assumed Lipschitz bound on $u$.
\subsection{Identification of the \texorpdfstring{measure $\Delta u$}{Laplacian measure}}\label{sec:distributional}
\begin{theorem}\label{thm:distributional}
	Let $(u, \chi)$ be a variational solution on $\W$. Then $\p \{u > 0\}$ is countably $\cH^{n-1}$-rectifiable, has locally finite $\cH^{n-1}$-measure, and
	\[
		\Delta u = \cH^{n-1}\mres \p^* \{ u > 0\} + 2 \a(n) \sqrt{H_x(u; 0+)} \cH^{n-1}\mres \Sigma^H.
	\]
\end{theorem}

\begin{proof}
	The measure $\Delta u$ is supported on $\p \{u > 0\}$, and applying Corollary \ref{cor:regular} to $u$ on the open set $\W \sm \Sigma^H$ gives
	\[
		\Delta u \mres (\W \sm \Sigma^H) = \cH^{n-1}\mres \p^* \{ u > 0\}.
	\]
	We also have the estimate from Lemma \ref{lem:BV}
	\[
		\Delta u(B_r(x)) \leq C(C_V) r^{n-1}
	\]
	for any $B_{2r}(x) \cc \W$, which implies that $ \cH^{n-1}((\p \{ u > 0\} \sm \Sigma^H) \cap \W') = \cH^{n-1}(\p^* \{ u > 0\} \cap \W') < \infty$ for any $\W' \cc \W$. On the other hand, by Theorem \ref{thm:main}, $\cH^{n-1}(\Sigma^H \cap \W') < \infty$, so $\cH^{n-1}(\W' \cap \p \{u > 0\}) < \infty$. Also we have that $\Sigma^H$, and so all of $\p \{u > 0\}$, is countably $\cH^{n-1}$-rectifiable.

	By \cite{F}[3.2.19] and the Radon-Nikodym theorem, if we take the density
	\[
		\theta(x) = \limsup_{r \to 0} \frac{\Delta u(B_r(x))}{\w_{n-1}r^{n-1}} \leq C(C_v)
	\]
	where $\w_{n-1}$ is the volume of the unit ball in $\R^{n-1}$, then
	\[
		\Delta u = \theta \cH^{n-1}\mres \p\{ u > 0 \}.
	\]
	Therefore it suffices to show that for $\cH^{n-1}$-a.e. point $x \in \Sigma^H$, we have $$\theta(x) = 2 \a(n) \sqrt{H_x(u; 0+)}.$$

	Applying Lemma \ref{lem:singular1density}, for $\cH^{n-1}$-a.e. point $x \in \Sigma^H$ we have that $v_{x, r}(y) = \frac{u(x + r y)}{r \sqrt{H_x(u; r)}} \rightarrow v(y) = \a(n) |y_n|$ in $W^{1, 2}_{\mathrm{loc}}(\R^n)$, for some choice of coordinates on $\R^n$. Then
	\[
		\frac{\Delta u(B_r(x))}{r^{n-1} \sqrt{H_x(u; r)}} = \Delta v_{x, r}(B_1) \rightarrow \Delta v(B_1) = 2 \a(n) \cH^{n-1}(B_1 \cap \{y_n = 0\}) = 2 \a(n) \w_{n-1}
	\]
as $r\to 0$,
	using that $\Delta v (\p B_1) = 0$ and the fact that $\Delta v_{x, r} \rightharpoonup \Delta v$ in the weak-$*$ sense as measures as $r\to 0$, and computing the Laplacian explicitly in the last step. Multiplying both sides by $\sqrt{H_x(u; 0+)} \geq 0$ leads to
	\[
		\theta(x) = \limsup_{r \to 0} \frac{\Delta u(B_r(x))}{\w_{n-1}r^{n-1}} = 2 \a(n) \sqrt{H_x(u; 0+)},
	\]
	completing the proof.
\end{proof}

\begin{remark}
	Along $\p^* \{u > 0\}$, an elementary computation shows that $\a(n)\sqrt{H_x(u; 0+)} = \frac{1}{2}$, so the conclusion of Theorem \ref{thm:distributional} could be rewritten as
	\[
		\Delta u = 2 \a(n) \sqrt{H_x(u; 0+)} \cH^{n-1}\mres \p \{u > 0\}.
	\]
\end{remark}

\begin{proof}[Proof of the Main Theorem \ref{thm:intro}:]
(i) follows from Theorem \ref{thm:chi}.\\
(ii) has been shown in Theorem \ref{thm:distributional}.\\
(iii) follows from a combination of Theorem \ref{thm:distributional}
and Lemma \ref{lem:singular1density}.\\
(iv) follows from Lemma \ref{lem:singular1density} and Theorem \ref{thm:main}.
\end{proof}

\section{Application to a compactness result for classical solutions}\label{sec:compactness}

In this section we present a simple example of how to use the theory developed to study limiting behavior of related problems.

\begin{definition}
	Given an open set $\W$, we say a Lipschitz continuous function $u : \W \rightarrow [0, \R)$ is a \emph{classical solution}  if $u$ is harmonic on $\{u > 0\}$, the set $\p\{u > 0\}$ is locally given as a graph of a smooth function, and ---provided that $\nu_x$ is a unit normal vector to $\p \{u > 0\}$ at $x$ such that $u(x + t \nu) > 0$ for $t \in (0, t_0)$--- $\lim_{t \to 0+}\frac{u(x + t \nu) - u(x)}{t} = 1$.
\end{definition}

Note that if $u$ is a classical solution, then $(u, \chi_{\{u > 0\}})$ is a variational solution; \eqref{eq:variationalformula} can be verified through integration by parts.

\begin{proposition}\label{prop:lip}
	Let $u$ be a classical solution on $B_1$, and assume that $u(0) = 0$. Then for $r \in (0,1)$,
	\[
		\sup_{B_r} |\n u| \leq 1 + \w(r)
	\]
	for some continuous nondecreasing function $\w : [0, 1)\to \R$ with $\w(0+) = 0$, depending only on $n$.
\end{proposition}

The proof of this may be found in \cite{AC}[Theorem 6.3].

\begin{lemma}\label{lem:babydensityest}
	Let $(u, \chi_{\{u > 0\}})$ be a variational solution on $B_r(x)$ with $x \in \p \{u > 0\}$. Then
	\[
		|\{u > 0\} \cap B_{r}(x)| \geq c r^n,
	\]
	where $c = c(n, C_V) > 0$.
\end{lemma}

\begin{proof}
	We show this with $x = 0$ and $r = 1$. By Lemma \ref{lem:density0char}, we have that $M_0(0+) \geq |B_{1}|/2$, so
	\[
		\frac{|B_1|}{2} \leq M_0(0+) \leq M_0(1) \leq \int_{B_1} \left(|\n u|^2 + \chi_{\{u > 0\}}\right) \leq (1 + C_V^2)|\{u > 0\} \cap B_1|.
	\]
\end{proof}

The following theorem proves Theorem \ref{thm:compact_intro}.
\begin{theorem}\label{classcomp}
	Let $u_k$ be a sequence of classical solutions on $B_1$, with
	\[
		\sup_k \|\n u_k \|_{C^{0,1}(B_1)} < \infty.
	\]
	Then along a subsequence, $u_k \rightarrow u \in C^{0, 1}$ locally uniformly, $(u, \chi_{\{u > 0\}})$ is a variational solution, and for $\cH^{n-1}$-a.e. $x \in \Sigma^H(u)$,
	\[
		\a(n)\sqrt{H_x(u; 0+)} \leq 1.
	\]
	Moreover, either $u \equiv 0$ or $\chi_{\{u_k > 0\}} \rightarrow \chi_{\{u > 0\}}$ in $L^1_{\mathrm{loc}}(B_1)$ and $\p \{u_k > 0\} \rightarrow \p \{u > 0\}$ locally in $B_1$ in Hausdorff topology.
Most importantly,
$\p \{u > 0\}$ is countably $\cH^{n-1}$-rectifiable, has locally in $B_1$ 
finite $\cH^{n-1}$-measure, and
	\[
		\Delta u = \cH^{n-1}\mres \p^* \{ u > 0\} + 2 \a(n) \sqrt{H_x(u; 0+)} \cH^{n-1}\mres \Sigma^H.
	\]
\end{theorem}

The hypothesis is satisfied, for example, for every sequence $u_k$ of classical solutions on $B_{1 + \d}$ with $u_k(0) = 0$.

\begin{proof}
	Take a subsequence along which $u_k \rightarrow u$ locally uniformly and weakly in $W^{1, 2}$, and so $\Delta u_k \rightarrow \Delta u$ in the sense of distributions as $k\to\infty$. The limit $u$ must be Lipschitz continuous and harmonic where positive. It is then straightforward to check that $u_k \rightarrow u$ strongly in $W^{1, 2}_{\mathrm{loc}}(B_1)$ as $k\to\infty$:
	\[
		\int \eta |\n u|^2 = - 2 \int u \n u \cdot \n \eta = \lim_{k\to\infty} - 2 \int u_k \n u_k \cdot \n \eta = \lim \int \eta |\n u_k|^2
	\]
	for any $\eta \in C_c^\infty(B_1)$. Using Lemma \ref{lem:BV}, we have that $\chi_{\{u_k > 0\}} \rightarrow \chi$ in $L^1_{\mathrm{loc}}(B_1)$ as $k\to\infty$ and a.e. along a further subsequence, with $\chi \in \{0, 1\}$ almost everywhere. We may then pass \eqref{eq:variationalformula} to the limit, showing that $(u, \chi)$ is a variational solution.

	By Theorem \ref{thm:chi}, $(u, \chi_{\{u > 0\}})$ is also a variational solution, and unless $u \equiv 0$, $\chi_{\{u_k > 0\}} \rightarrow \chi_{\{u > 0\}}$ locally in $L^1$.
Moreover,
for each $x \in \p \{u > 0\}$, we have that $x = \lim_{k\to\infty} x_k$ for some $x_k \in \p \{u_k > 0\}$. Indeed, were this not true, for some $\d$ we would either have $u_k = 0$ on $B_\d(x)$, in which case $u = 0$ on $B_\d(x)$ and $x \notin \p \{u > 0\}$, or $u_k > 0$ on $B_\d(x)$, in which case $\sup_{B_{\d/2}(x)} u_k \rightarrow 0$ by Harnack inequality and again $x\notin \p \{u > 0\}$.
It remains to show that for any $x = \lim x_k \in B_1$ with $x_k \in \p \{u_k > 0\}$, we have that $x \in \p \{u > 0\}$. 
By uniform convergence of $u_k$, $u(x)=0$.
On the other hand, if $x \in \{u = 0\}^\circ$, then there is a ball $B_\d(x)$ with $u = 0$ on it, and so from the $L^1$-convergence of $\chi_{\{u_k > 0\}}$ we have $|B_\d(x) \cap \{u_k > 0\}| \rightarrow 0$. For any $\eta$ and large $k$, then, $|B_{\d/2}(x_k) \cap \{u_k > 0\}| < \eta \d^n$. But from Lemma \ref{lem:babydensityest} (and Proposition \ref{prop:lip}) we have that $|B_{\d/2}(x_k) \cap \{u_k > 0\}| > c(n) \d^n$, which is a contradiction for sufficiently small $\eta$.

	Take an $x\in \Sigma^H$ at which $\frac{u(x + r y)}{r \sqrt{H_x(u; r)}} \rightarrow \a(n) |y_n|$ as $r\to 0$ (this happens at a.e. $x \in \Sigma^H$ by Lemma \ref{lem:singular1density}). Applying Proposition \ref{prop:lip} on $B_\d(x_k)$ for $x_k \rightarrow x$ with $u_k(x_k) = 0$ and taking limits, we have that for all $\d$ small,
	\[
		\sup_{B_\d(x)} |\n u| \leq 1 + \w(\d).
	\]
	Then
	\[
		\limsup_{r \to 0 } \left\Vert\frac{u(x + r \cdot)}{r}\right\Vert_{C^{0, 1}(\overline{B_1})} \leq 1,
	\]
	so $\a(n)\sqrt{H_x(u; 0+)}  \leq 1$.

The last statements follow from Theorem \ref{thm:distributional}.
\end{proof}

\begin{remark} An analogous compactness theorem can be proved for limits of arbitrary variational solution $u_k$ (uniformly bounded in $C^{0,1}$), except for the conclusion $\a(n)\sqrt{H_x(u; 0+)} \leq 1$. If the $u_k$ are also viscosity solutions, then $\a(n)\sqrt{H_x(u; 0+)} \leq 1$ can be recovered using Lemma \ref{lem:singular1density}: at almost every point $x \in \Sigma^H$ where $\a(n)\sqrt{H_x(u; 0+)} > 0$, we have that $u(x + r y)/r \rightarrow \a(n)\sqrt{H_x(u; 0+)} |y_n|$ strongly in $W^{1,2}$ and \emph{locally uniformly}. It follows that the limit function $\a(n)\sqrt{H_x(u; 0+)} |y_n|$ is also a viscosity solution, and one can check this is only true if $\a(n)\sqrt{H_x(u; 0+)} \leq 1$.

	As the properties of being a viscosity or variational solution are preserved under limits, and viscosity solutions satisfy a Bernstein-type gradient bound (see \cite{CS}), this is a useful class of solutions to consider.
\end{remark}

\section{Application to a singular perturbation problem}\label{sec:perturbation}
Consider the singular perturbation
problem
\begin{equation}\label{epsprob}
\Delta u_\epsilon=\beta_\epsilon(u_\epsilon) \> \textrm{ in } \Omega,
\end{equation}
where $\Omega$ is a bounded domain in $\R^n$, $\epsilon>0$ and
$\beta_\epsilon(s)=\frac{1}{\epsilon}\beta(\frac{s}{\epsilon})$.
Here $\beta$ is a Lipschitz continuous function such that $\beta>0$ in
$(0, 1)$, $\beta\equiv 0$ outside $(0, 1)$ and
$\int_0^1\beta(s)\> ds=\frac{1}{2}$.
Note that we do not assume that $u_\epsilon$ is a minimizer or a stable solution.

The problem arises in the mathematical analysis of
equidiffusional flames (see \cite{bcn}, \cite{buckmaster}),
in which case $\epsilon$ is proportional to the inverse of the
activation energy.

Formally, as $\epsilon\to 0$, the solutions $u_\epsilon$
converge to a solution $u$ of the free boundary
problem 
\begin{equation}\label{fbp}
\begin{array}{ll}
\Delta u=0  &\hbox{in}\ \Omega\setminus\partial\{u>0\},\\
u=0, \ (u^+_\nu)^2-(u^-_\nu)^2=1 & \hbox{on}\
\partial\{u>0\}.
\end{array}
\end{equation}

Luis Caffarelli established a locally
uniform Lipschitz estimate for bounded solutions $u_\epsilon$
(see \cite{cafsing}).
In \cite{lederman}, C. Lederman and
N. Wolanski proved that $u$ is a viscosity solution of 
(\ref{fbp}). They also proved that
$u$ satisfies the free boundary condition in a pointwise sense
at non-degenerate free boundary points at which there is an inward unit normal of
$\{u>0\}$ in the measure-theoretic sense. 

However, uniformly bounded solutions may converge to limits exhibiting 
degenerate singular free boundary points (see \cite{weisszhang}).

In \cite{W03} the second author showed that, in a special case
of the time-dependent problem treated therein,
\[
		\Delta u = \cH^{n-1}\mres \p^* \{ u > 0\} + 2 \theta(x) \cH^{n-1}\mres \Sigma_{**}
+ \lambda,
	\]
where 
$\lambda$ is a measure supported on the degenerate singular set.
However, the question whether $\lambda$ is nonzero and
whether its support can be a large set has been an open problem
since \cite{W03}, 
except for the case of two dimensions where David Jerison and 
the second author applied (see \cite[Introduction]{W03}) a result by Tom Wolff on the harmonic
measure \cite{wolff} to show that, under mild assumptions to the set $\{ u > 0\}$,
$\lambda$ is zero and the topological free boundary is contained in a set
of $\sigma$-finite length.
In higher dimensions, the issue has been a challenge in view of
examples of extremely irregular harmonic measures such as \cite[Example 3]{wolff2}.

Here we show the following, proving Theorem \ref{thm:perturb_intro}.
\begin{theorem}\label{singpert}
Let $(u_\epsilon)_{\epsilon\in (0,1)}$ be a uniformly bounded family of solutions
of \eqref{epsprob} on $B_1$.
Then each limit $u$ as $\epsilon\to 0$ along a subsequence
satisfies the following:
$(u,\chi_{\{u>0\}})$ is a variational solution,
 $\p \{u > 0\}$ is countably $\cH^{n-1}$-rectifiable, has locally in $B_1$ 
finite $\cH^{n-1}$-measure, and
	\[
		\Delta u = \cH^{n-1}\mres \p^* \{ u > 0\} + 2 \a(n) \sqrt{H_x(u; 0+)} \cH^{n-1}\mres \Sigma^H;\]
moreover,  $\a(n) \sqrt{H_x(u; 0+)}\leq 1$ for $\cH^{n-1}$-a.e. $x\in \Sigma^H$. 
\end{theorem}
\begin{remark}\label{rem:harmonicmeasure}
It follows that the harmonic measure of the set $\{ u > 0\}$ is supported on a 
countably $\cH^{n-1}$-rectifiable set of locally 
finite $\cH^{n-1}$-measure. Under additional mild assumptions to the set $\{ u > 0\}$
which ensure that a boundary Harnack inequality holds, this harmonic measure and
the Laplacian measure $\Delta u$ are comparable, ruling out even more 
pathological behaviors observed for harmonic measure. For example, \cite{wolff2} constructs domains $\W$ with harmonic measure not absolutely continuous with respect to $\cH^{n-1}\mres \W$, while \cite{Wu} does so even for domains with rectifiable boundaries with locally finite $\cH^{n-1}$ measure.

The remark remains true for all variational solutions, including
the limits of classical solutions of the Bernoulli problem 
in Section \ref{sec:compactness}.
\end{remark}

In order to apply the main theorem of our paper,
it is sufficient to show that each limit is a variational
solution. To this end, we cite several known results:
\begin{lemma}[\cite{cafsing}]\label{uniforml Lipschitz bound}
Let $(w_\epsilon)_{\epsilon\in (0,1)}$ be a family of solutions to $\Delta
w_\epsilon=\beta_\epsilon(w_\epsilon)$ in a domain $\Omega\subset
\mathbb{R}^n$ such that $||w_\epsilon||_{L^\infty(\Omega)}\leq C$ for some
constant $C$. Let $K\subset\Omega$ be a compact set and let $\tau>0$ be such
that $B_\tau(x^0)\subset\Omega$ for every $x^0\in K$. Then there
exists a constant $L=L(\tau, C)$, such that
$$
|\nabla w_\epsilon(x)|\leq L \ \ \ \mathrm{for}\  x\in K.
$$
\end{lemma}
\begin{lemma}[see for example {\cite[Proposition 4.3]{weisszhang}}]\label{conv}
Let $(u_\epsilon)_{\epsilon\in (0,1)}$ be a uniformly bounded family of solutions
of \eqref{epsprob} on $B_1$.
There exist a sequence $(u_{\epsilon_i})_{i\in \N}$ and a locally in $B_1$
Lipschitz continuous function $u$ such that\\
1) $u_{\epsilon_i}\rightarrow u$ locally uniformly in $B_1$, \\
2) $u_{\epsilon_i}\rightarrow u$ in $W^{1, 2}_{\mathrm{loc}}(B_1)$, \\
3) $u$ is harmonic in $B_1\setminus\partial\{u>0\}$, \\
4) $\Delta u_{\epsilon_i}\rightarrow \mu$ as measures on $B_1$ as $i\to\infty$; here
$\mu$ is a locally finite non-negative measure supported on the free boundary
$\partial\{u >0\}$.Therefore $$ \Delta u=\mu\qquad \hbox{in}\
B_1.$$
\end{lemma}
Let
$\mathcal{B}_{\epsilon}(z)=\int_0^z\beta_{\epsilon}(s)\> ds$ and $\chi_\epsilon(x)=2\mathcal{B}_\epsilon(u_\epsilon(x))$. It follows that $0\leq\chi_{\epsilon}(x)\leq
1$, 
and the uniform Lipschitz estimate for $u_\epsilon$ implies the relative compactness of
$\chi_{\epsilon}$ in $L_{\mathrm{loc}}^1(B_1)$:
\begin{lemma}[{\cite[Proposition 4.4]{weisszhang}}]
Let $(u_\epsilon)_{\epsilon\in (0,1)}$ be a uniformly bounded family of solutions
of \eqref{epsprob} on $B_1$. Then
$(\chi_{\epsilon})_{\epsilon \in (0,1)}$ is precompact in $L^1(D)$ for each $D\subset\subset B_1$.
\end{lemma}
Thus we may assume that
$\chi_{\epsilon_i}(x)\rightarrow\chi(x)$ locally in $L^1$ as $i\to\infty$.
\begin{lemma}[{\cite[Proposition 4.5]{weisszhang}}]
Let $(u_\epsilon)_{\epsilon\in (0,1)}$ be a uniformly bounded family of solutions
of \eqref{epsprob} on $B_1$
and let $\chi$ be the limit chosen above. Then
$\chi(x)\in\{0, 1\}$ for a.e. $x\in B_1$.
\end{lemma}
\begin{lemma}[{\cite[Lemma 6.1]{weisszhang}}]\label{domvar}
Let $w$ be a solution of
$$
\Delta w=\beta_\epsilon(w)\textrm{ in }
\Omega\subset\mathbb{R}^n\textrm{ and let } \chi_\epsilon(x) = 2\mathcal{B}_\epsilon(w(x)).
$$
Then 
$$
\int_\Omega\left(|\nabla w|^2 \dvg\phi-2\nabla w D\phi\nabla
w+\chi_\epsilon\dvg\phi\right)\> dx=0 
$$ 
for $\phi\in C_c^1(\Omega;\mathbb{R}^n)$.
\end{lemma}
\begin{theorem}
Let $(u_\epsilon)_{\epsilon\in (0,1)}$ be a uniformly bounded family of solutions
of \eqref{epsprob} on $B_1$, and let
$u$ be the limit chosen above. Then
$(u, \chi_{\{u > 0\}})$ is a variational solution in $B_1$.
\end{theorem}
\proof
By the strong convergence Lemma \ref{conv},
$(u, \chi)$ is a variational solution in $B_1$.
But then the result follows from Theorem \ref{thm:chi}.
\qed
\section*{Acknowledgments}
We would like to thank Jonas Hirsch for helpful discussions, and the Mittag-Leffler Institute for its hospitality during the program ``Geometric Aspects of Nonlinear Partial Differential Equations." DK was supported by NSF DMS grant 2247096. We are grateful to the anonymous referees for helpful suggestions and comments.

\bibliographystyle{plain}
\bibliography{compactness_bernoulli_020723}

\end{document}